\documentclass[11pt,fleqn]{article}

\usepackage{algorithm}
\usepackage{algpseudocode}

\usepackage{amsmath}
\usepackage{amsthm}
\usepackage{amsfonts}
\usepackage{amssymb}

\usepackage[dvipsnames]{xcolor}

\definecolor{tue_red}{HTML}{C81919}

\definecolor{tue_red}{HTML}{C81919}
\definecolor{tue_dark_blue}{HTML}{101073}
\definecolor{tue_blue}{HTML}{0066CC}
\definecolor{tue_cyan}{HTML}{00A2DE}
\definecolor{tue_green}{HTML}{84D200}
\definecolor{tue_yellow}{HTML}{CEDF00}

\definecolor{accent}{gray}{0.95}

\definecolor{color1}{RGB}{0, 121, 178}
\definecolor{color2}{RGB}{255, 124, 37}
\definecolor{color3}{RGB}{37, 160, 55}
\definecolor{color4}{RGB}{220, 32, 44}
\definecolor{color5}{RGB}{147, 104, 186}
\definecolor{color6}{RGB}{143, 85, 76}
\definecolor{color7}{RGB}{230, 119, 192}
\definecolor{color8}{RGB}{127, 127, 127}
\definecolor{color9}{RGB}{192, 188, 55}
\definecolor{color10}{RGB}{0, 191, 206}

\usepackage[margin = 1in, twoside]{geometry}

\usepackage{microtype}

\usepackage[UKenglish]{babel}
\usepackage[UKenglish]{isodate}

\PassOptionsToPackage{hyphens}{url}
\usepackage[hypertexnames = false, pdftex]{hyperref}
\hypersetup{colorlinks = true, linkcolor = blue, citecolor = blue, urlcolor = blue, linktocpage}

\usepackage[numbers, square, comma]{natbib}
\usepackage{doi}

\usepackage{enumitem}

\usepackage{graphicx}
\usepackage{subcaption}

\usepackage{comment}

\usepackage{booktabs, tabu}

\usepackage{multicol}

\usepackage{bbm}

\usepackage{float}

\usepackage[T1]{fontenc}
\usepackage{bm}
\usepackage{mathrsfs}

\newcommand{\Vn}{\ensuremath{{\fV_n}}}

\newcommand{\Wm}{\ensuremath{{\fW_m}}}

\newcommand{\fTheta}{\ensuremath{\bm{\Theta}}}
\newcommand{\fGamma}{\ensuremath{\bm{\Gamma}}}

\newcommand{\Hn}{\ensuremath{{\fH_n}}}

\newcommand{\eps}{\ensuremath{\varepsilon}}

\newcommand{\id}{\ensuremath{\textrm{Id}}}

\usepackage{mathtools}
\usepackage{suffix}
\usepackage{stmaryrd} %

\usepackage{tensor}

\DeclarePairedDelimiter{\prt}{(}{)}

\DeclarePairedDelimiter{\norm}{\lVert}{\rVert}
\DeclarePairedDelimiter{\inner}{\langle}{\rangle}
\DeclarePairedDelimiter{\set}{\{}{\}}

\newcommand{\Norm}[1]{\left\lVert#1\right\rVert}

\let \oldforall \forall
\let \forall \undefined
\DeclareMathOperator{\forall}{\oldforall}

\let \oldexists \exists
\let \exists \undefined
\DeclareMathOperator{\exists}{\oldexists}

\let \oldtext \text
\renewcommand{\text}[1]{~\oldtext{#1}~}

\newcommand{\cond}{\, : \,}
\newcommand{\st}{\text{s.t.}}

\DeclareMathOperator{\Lin}{Lin}

\DeclareMathOperator*{\argmin}{arg \, min}

\usepackage{ifthen}
\newlength{\leftstackrelawd}
\newlength{\leftstackrelbwd}
\def \leftstackrel#1#2{\settowidth{\leftstackrelawd}%
  {${{}^{#1}}$} \settowidth{\leftstackrelbwd}{$#2$}%
  \addtolength{\leftstackrelawd}{- \leftstackrelbwd}%
  \leavevmode \ifthenelse{\lengthtest{\leftstackrelawd>0pt}}%
  {\kern-.5 \leftstackrelawd}{} \mathrel{\mathop{#2} \limits^{#1}}}

\DeclareMathOperator{\grad}{grad}

\DeclareMathOperator{\vspan}{span}

\DeclareMathOperator{\dom}{dom}

\newcommand{\proj}[1]{\ensuremath{\rP_{#1}}}

\newcommand{\tangent}[2]{\ensuremath{\rT_{#1}{#2}}}

\usepackage{cleveref}

\newbool{isrelease}

\newcounter{review}

\makeatletter

\newcommand \listreviewname{List of Reviews}
\newcommand \listofreviews{\section*{\listreviewname} \addcontentsline{toc}{section}{List of Reviews} \@starttoc{tor}}
\makeatother

\newcommand{\bA}{\ensuremath{\mathbb{A}}}

\newcommand{\bC}{\ensuremath{\mathbb{C}}}

\newcommand{\bG}{\ensuremath{\mathbb{G}}}

\newcommand{\bI}{\ensuremath{\mathbb{I}}}
\newcommand{\bJ}{\ensuremath{\mathbb{J}}}

\newcommand{\bM}{\ensuremath{\mathbb{M}}}

\newcommand{\bR}{\ensuremath{\mathbb{R}}}
\newcommand{\bS}{\ensuremath{\mathbb{S}}}

\newcommand{\bU}{\ensuremath{\mathbb{U}}}

\newcommand{\bW}{\ensuremath{\mathbb{W}}}

\newcommand{\bY}{\ensuremath{\mathbb{Y}}}
\newcommand{\bZ}{\ensuremath{\mathbb{Z}}}

\newcommand{\cA}{\ensuremath{\mathcal{A}}}
\newcommand{\cB}{\ensuremath{\mathcal{B}}}
\newcommand{\cC}{\ensuremath{\mathcal{C}}}

\newcommand{\cH}{\ensuremath{\mathcal{H}}}

\newcommand{\cJ}{\ensuremath{\mathcal{J}}}

\newcommand{\cL}{\ensuremath{\mathcal{L}}}
\newcommand{\cM}{\ensuremath{\mathcal{M}}}

\newcommand{\cO}{\ensuremath{\mathcal{O}}}
\newcommand{\cP}{\ensuremath{\mathcal{P}}}

\newcommand{\cU}{\ensuremath{\mathcal{U}}}
\newcommand{\cV}{\ensuremath{\mathcal{V}}}

\newcommand{\rP}{\ensuremath{\mathrm{P}}}

\newcommand{\rT}{\ensuremath{\mathrm{T}}}

\newcommand{\rw}{\ensuremath{\mathrm{w}}}

\newcommand{\rz}{\ensuremath{\mathrm{z}}}

\newcommand{\fD}{\ensuremath{\bm{\mathrm{D}}}}

\newcommand{\fG}{\ensuremath{\bm{\mathrm{G}}}}
\newcommand{\fH}{\ensuremath{\bm{\mathrm{H}}}}

\newcommand{\fL}{\ensuremath{\bm{\mathrm{L}}}}

\newcommand{\fT}{\ensuremath{\bm{\mathrm{T}}}}

\newcommand{\fV}{\ensuremath{\bm{\mathrm{V}}}}
\newcommand{\fW}{\ensuremath{\bm{\mathrm{W}}}}

\renewcommand{\d}{\mathrm{d}}

\newcommand{\dt}{\d t}

\DeclareFontFamily{U}{matha}{\hyphenchar\font45}
\DeclareFontShape{U}{matha}{m}{n}{
      <5> <6> <7> <8> <9> <10> gen * matha
      <10.95> matha10 <12> <14.4> <17.28> <20.74> <24.88> matha12
      }{}
\DeclareSymbolFont{matha}{U}{matha}{m}{n}

\DeclareMathSymbol{\Lt}{3}{matha}{"CE}
\DeclareMathSymbol{\Gt}{3}{matha}{"CF}

\usepackage{tikz}

\usetikzlibrary{backgrounds}

\usepackage{pgfplots}
\pgfplotsset{compat = 1.18}

\usepgfplotslibrary{groupplots}

\usepackage{pgfplotstable}

\newcommand{\Jcal}{\ensuremath{\mathcal{J}}}
\newcommand{\Mcal}{\ensuremath{\mathcal{M}}}

\newcommand{\Ucal}{\ensuremath{\mathcal{U}}}

\newcommand{\Hcal}{\mathcal{H}}

\newcommand{\Rbb}{\ensuremath{\mathbb{R}}}
\newcommand{\Nbb}{\ensuremath{\mathbb{N}}}

\newcommand{\Sbb}{\ensuremath{\mathbb{S}}}

\newcommand{\Cbb}{\ensuremath{\mathbb{C}}}
\newcommand{\Gbb}{\ensuremath{\mathbb{G}}}

\newcommand{\inprod}[3]{\left<{#1},{#2}\right>_{#3}}
\newcommand{\inprodV}[2]{\inprod{#1}{#2}{\Hs}}

\newcommand{\nr}{\ensuremath{n}} 
\newcommand{\nm}{\ensuremath{m}} 

\newcommand{\Vtwon}{\ensuremath{\mathbf{H}_{2\nr}}}
\newcommand{\Wtwom}{\ensuremath{\mathbf{W}_{2\nm}}}
\newcommand{\Hs}{{\fH}} 
\newcommand{\Hhat}{\widehat{\fH}}

\newtheorem{theorem}{Theorem}

\newtheorem{remark}{Remark}

\newtheorem{lemma}{Lemma}

\title{Symplectic Filtering of Hamiltonian Dynamics\\ with Moving Sensors}
\author{Olga Mula, Cecilia Pagliantini\thanks{C.P. acknowledges the MIUR Excellence Department Project awarded to the Department of Mathematics, University of Pisa, CUP I57G22000700001 and the INdAM-GNCS.}, Federico Vismara\thanks{F.V. acknowledges funding by the European Union (ERC, COCOA, 101170147).}}
\date{}

\begin{document}

\maketitle

\begin{abstract}
  While structure-preserving methods for the forward simulation of Hamiltonian systems are well established, the mathematical foundations of their data assimilation and filtering counterparts remain comparatively less developed, particularly regarding rigorous accuracy guarantees. To address this gap, we propose an online filtering algorithm to reconstruct an unknown function $u^\dagger$ from finitely many measurements, assuming that $u^\dagger$ solves a parametric Hamiltonian PDE with unknown inputs. The method uses low-dimensional symplectic approximation spaces that evolve in time according to the PDE model and informed by the measurement data. The method can also be coupled with a dynamical sensor-placement strategy designed to optimize reconstruction stability. We derive error bounds in terms of stability and best approximation, and show that the reconstruction remains in a symplectic space and preserves energy up to the approximation error. Numerical experiments demonstrate the effectiveness of the method and the benefits of dynamic over static sensor placement.
\end{abstract}

\section{Introduction}
\label{sec:intro}

Hamiltonian dynamics provide a fundamental framework for modeling conservative systems arising in areas such as celestial mechanics, molecular dynamics, and plasma physics. Their evolution is governed by an energy functional and a symplectic structure, which gives rise to characteristic geometric properties such as conservation of invariants and phase-space volume preservation. These features are essential for accurately capturing the long-term behavior of the system, but they pose significant challenges for numerical approximation. While structure-preserving methods for forward simulation have been extensively studied, filtering and data assimilation for Hamiltonian systems remain largely unexplored. Standard filtering and estimation techniques may be prohibitively expensive, as they typically require repeated forward solves, and may fail to preserve the underlying geometry, leading to physically inconsistent or unstable reconstructions. The main goal of this paper is therefore to develop the mathematical foundations of a computationally efficient filtering strategy that preserves the intrinsic geometric structure of Hamiltonian dynamics.

The introduction is organized as follows. In \Cref{sec:filtering}, we define the mathematical framework for the filtering problem under consideration. In \Cref{sec:idealized-strategy}, we give an overview of the reconstruction paradigm developed in the paper. Sections \ref{sec:novelty-and-plan} and \ref{sec:earlier-works} explain the main contributions of the strategy in connection with earlier works, and outline the plan of the rest of the paper.

\subsection{The filtering problem}
\label{sec:filtering}

Our goal is to reconstruct a time-dependent function $u^\dagger(t)$ from partial observations. For every $t\in \fT:=\bR_+$, $u^\dagger(t)$ belongs to a Hilbert space $\fH$ defined over a spatial domain $\fD\subset \bR^d$, and with norm $\norm{\cdot}_{\fH}$ induced by the inner product $\inner{\cdot,\cdot}_{\fH}$. We assume in the following that the evolution is smooth enough so that $u^\dagger \in \cC^1(\fT; \fH)$. At every time $t\in \fT$, $u^\dagger(t)$ is only partially known through $m$ observations of the form
\begin{equation*}
  \rz_i(t) = \ell_i(u^\dagger(t)), \qquad i = 1,\dots, m,
\end{equation*}
where $\ell_i: \fH \rightarrow \bR$ are known, independent, continuous linear functionals from $\fH'\coloneqq \Lin(\fH;\bR)$, the dual space of $\fH$. The functionals $\{\ell_i\}_{i=1}^m$ often represent the response of a physical measurement device, but they can have a different interpretation depending on the application. In the following, we gather the observations in a vector,
\begin{equation}
  \label{eq:obs}
  \rz(t) = \ell(u^\dagger(t)) \in \bR^m,
\end{equation}
where $\ell=(\ell_1,\dots,\ell_m)\in \Lin(\fH; \bR^m)$. In fact, at every time $t\in \fT$, we have access to the full observation history until that moment, which we denote as
$$
  \underline{\rz}(t) = \{ \rz(s) \cond s \in [0,t]\}.
$$
Since the functionals $\ell_i$ are linearly independent, the knowledge of $\underline{\rz}(t)$ is equivalent to knowing
$$
  \underline{\rw}(t) = \{ \rw(s) \cond s \in [0,t]\},
$$
where, for every $t\in\fT$,
$$
  \rw(t)\coloneqq \proj{\Wm}(u^\dagger(t))
$$
is the orthogonal projection of $u^\dagger(t)$ onto the observation space
$$
  \Wm\coloneqq \vspan\set{\rw_i}_{i=1}^m,
$$
and $\rw_i\in \fH$ are the Riesz representers of the linear functionals $\ell_i$, defined as
$$
  \ell_i(v)=\inner{\rw_i,v}_{\fH},\quad \forall v\in \fH.
$$

Recovering $u^\dagger(t)\in \fH$ from $\underline{\rw}(t)\in \Wm$ is an ill-posed problem whenever the dimension of $\fH$ exceeds $m$. Indeed, in that case, for any observation $\rw(t)\in \Wm$, there are infinitely many states $u(t)\in \fH$ such that $\proj{\Wm}(u(t))=\rw(t)$.
Thus, to recover the evolution $u^\dagger$ up to a guaranteed accuracy one has to combine the observations with some a priori information. In our case, we assume that $u^\dagger$ is the solution to an evolution problem of the form
\begin{equation}\label{eq:PDE}
  \begin{cases}
    \partial_t u = \cB_\theta(u)       & \quad (t,x)\in \fT\times \fD, \\
    u(\theta)(t=0, x) = u_0(\theta)(x) & \quad x\in\fD,
  \end{cases}
\end{equation}
where $\cB_\theta$ is a partial differential operator with appropriate boundary conditions. In our developments, $\cB_\theta$ will be a Hamiltonian operator, which we rigorously define later on in \Cref{sec:hamitonian-flows}. The operator depends on a parameter $\theta$ from a compact space ${\fTheta}$, which can be finite or infinite dimensional. We assume that, for each $\theta \in {\fTheta}$, problem \eqref{eq:PDE} is well-posed in the sense that there exists a unique solution $u(\theta)\in \cC^1(\fT; \fH)$.

In the context of the filtering problem, we then assume that $u^\dagger(t)=u(\theta^\dagger)(t, \cdot)$ for an unknown parameter $\theta^\dagger\in {\fTheta}$. Since $\theta^\dagger$ is unknown, the dynamics of $u^\dagger$ cannot be recovered by setting $\theta=\theta^\dagger$ in problem \eqref{eq:PDE}, and solving the PDE with classical discretization methods. Thus, our prior knowledge on $u^\dagger$ is only that it belongs to the set of trajectories
\begin{equation}
  \label{eq:set-param-sols}
  \cU \coloneqq \{ u(\theta) \in \cC^1(\fT; \fH) \text{ solution to \eqref{eq:PDE}} \cond \theta \in {\fTheta}  \} \; \subset \cC^1(\fT; \fH).
\end{equation}
To study details about the dynamics, it will be useful to slice the set $\cU$ in time as
\begin{equation*}
  \cU = \bigcup_{t\in\fT}\; \cU(t),
\end{equation*}
with
\begin{equation}\label{eq:solution-set-t}
  \cU(t) \coloneqq \{ u(\theta)(t,\cdot)\in \fH \cond \theta \in {\fTheta} \}\;\subset\fH.
\end{equation}

To summarize, here is the task and our assumptions:
\paragraph{Task and assumptions.} Recover the dynamics of $u^\dagger$ assuming that:
\begin{enumerate}
  \item[(A1)] $u^\dagger \in \cU \subset \cC^1(\fT; \fH)$, which is equivalent to the existence of $\theta^\dagger\in {\fTheta}$ such that $u^\dagger(t) = u(\theta^\dagger)(t,\cdot)\in\fH$ for all $t\in \bR_+$;
  \item[(A2)] $\theta^\dagger$ is unknown, but, for every $t\in \bR_+$, we are given the evolution of the observations until time $t$, which is given by $\underline{\rz}(t)$ and $\underline{\rw}(t)$;
  \item[(A3)]
    for a given measure $\mu$ in the set $\cP_2({\fTheta})$ of probabilities with finite second moments, solutions to \eqref{eq:PDE} satisfy, for all $t>0$,
    \begin{equation*}
      u(\cdot)(t,\cdot)\in \fL\coloneqq L^2(({\fTheta},\mu); \fH)=\left\{f:\fTheta\rightarrow\fH\,:\,
      \norm{f}^2_{\fL} \coloneqq \int_{{\fTheta}} \norm{f(\theta)}^2_{\fH}\, \mu(\d\theta) <+\infty\right\}.
    \end{equation*}
\end{enumerate}

\begin{remark}
  In practice, observations are available only at discrete times. We nevertheless adopt a continuous-time formulation to exploit functional-analytic and differential-geometric connections between states and velocities. This yields a system of differential equations (\Cref{thm:dynamics-gamma} in \Cref{sec:filtering-scheme}), whose numerical integration naturally gives a time-discrete formulation (\Cref{app:discrete} in the Appendix) and allows us to assess the effect of the observation time step.
\end{remark}

\subsection{Main strategy for the recovery}
\label{sec:idealized-strategy}
Since, by (A1), $u^\dagger \in \cU$ and it is differentiable in time, its velocity satisfies
\begin{align*}
  \dot{u}^\dagger \in
  \dot{\cU}
   & \coloneqq \{ \dot{u}(\theta) \in \cC(\fT; \fH) \cond \theta \in {\fTheta}  \}
  \; \subset \cC(\fT; \fH),
\end{align*}
and we slice $\dot \cU$ just as we did with $\cU$, that is
\begin{equation*}
  \dot{\cU} = \bigcup_{t\in\fT}\; \dot{\cU}(t),
\end{equation*}
with
\begin{equation*}
  \dot{\cU}(t)
  \coloneqq \{ \dot{u}(\theta)(t,\cdot)\in \fH \cond \theta \in {\fTheta} \}
  = \{ \cB_\theta(u(\theta)(t,\cdot)) \in \fH \cond \theta \in {\fTheta}  \} \subset \fH.
\end{equation*}
In addition to this, from the state observations \eqref{eq:obs} we can also deduce velocity observations thanks to the continuity of $\ell$,
\begin{equation*}
  \dot{\rz}(t) =  \lim_{\delta t \to 0} \frac{\ell(u^\dagger(t+\delta t)) -\ell(u^\dagger(t))  }{\delta t} = \ell(\dot{u}^\dagger(t)) \in \bR^m,
\end{equation*}
together with the corresponding orthogonal projection
$$
  \dot{\rw}(t)=\proj{\Wm}(\dot{u}^\dagger(t)),\quad \forall t\in \bR_+.
$$
In fact, since for a given time $t>0$ we only have observations in the interval $[0, t]$, strictly speaking we only have access to the time-derivative from the left at that time,
\begin{equation*}
  \dot{\rz}_{-}(t) =  \lim_{\delta t \to 0, \delta t<0} \frac{\ell(u^\dagger(t+\delta t)) -\ell(u^\dagger(t))  }{\delta t} = \ell(\dot{u}_{-}^\dagger(t)) \in \bR^m.
\end{equation*}
To simplify the exposition, we work with $\dot{\rz}(t)$ at the expense of having a small inconsistency, because this detail will only be relevant at the level of numerically integrating the evolution equations that we will obtain.

Our strategy for reconstruction is as follows. For the initial time $t=0$, we want to identify a good approximation space $\Hn(0) \subset \fH$ to build a recovery algorithm
$$
  A(0): \Wm \to \Hn(0)
$$
such that $A(0)(\rw(0))$ approximates $u^\dagger(0)$ well. We say that $A(0)$ is near-optimal with respect to the recovery of all possible states in $\cU(0)$
if there exists $k(0)>0$ such that
$$
  E_{wc}(A(0), \cU(0), \Hn(0)) \leq k(0) \delta_{wc}(\cU(0), \Hn(0)),
$$
where
\begin{align*}
  E_{wc}(A(0), \cU(0), \Hn(0))
   & \coloneqq \max_{\theta\in {\fTheta}} \norm{u(\theta)(0,\cdot)-A(0)(\rw(0))}_{\fH},       \\
  \delta_{wc}(\cU(0), \Hn(0))
   & \coloneqq \max_{\theta\in {\fTheta}}\min_{v\in \Hn(0)}\norm{u(\theta)(0,\cdot)-v}_{\fH},
\end{align*}
are, respectively, the recovery error of the elements of $\cU(0)$ with algorithm $A(0)$ in the worst case sense, and the error of best approximation of elements of $\cU(0)$ with the space $\Hn(0)$. The same concepts can be defined in the average sense with respect to the probability distribution $\mu\in\cP({\fTheta})$ from our assumption (A3),
\begin{align*}
  E_{\mu}(A(0), \cU(0), \Hn(0))
   & \coloneqq \int_{{\fTheta}} \norm{u(\theta)(0,\cdot)-A(0)(\rw(0))}_{\fH} \,\mu(\d \theta),        \\
  \delta_{\mu}(\cU(0), \Hn(0))
   & \coloneqq \int_{{\fTheta}} \min_{v\in \Hn(0)}\norm{u(\theta)(0,\cdot)-v}_{\fH} \,\mu(\d \theta).
\end{align*}
They give rise to the concept of near-optimality in the average sense,
$$
  E_{\mu}(A(0), \cU(0), \Hn(0)) \leq k(0) \delta_{\mu}(\cU(0), \Hn(0)).
$$

For subsequent times $t>0$, our goal is to approximate well the velocity $\dot{u}^\dagger (t)$ from velocity observations $\underline{\dot{\rw}}(t)$, and then integrate in time. The motivation for working at the level of velocities and not the state itself is twofold: first, this will allow us to build a reconstructed state which is $\cC^1$ in time, thus coherent with the regularity of the exact solution $u^\dagger$. The second reason is connected to the preservation of certain properties of the flow, which usually comes expressed in the form of trajectories belonging to manifolds with well-defined tangent spaces where velocities need to be defined.
If we follow this rationale, we are after time-dependent approximation spaces $\Vn(t) \subset \fH$ to define an algorithm
$$
  \dot{A}(t): \Wm \to \Vn(t)
$$
such that $\dot{A}(t)(\dot{\rw}(t))$ approximates $\dot{u}^\dagger(t)$ well.
Similarly as before, we can define the concept of near-optimality for the class $\dot{\cU}(t)$ by defining $E_{\star}(\dot{A}(t),\dot{\cU}(t), \Vn(t))$ and $\delta_{\star}(\dot{\cU}(t), \Vn(t))$ for $\star=\{wc,\mu\}$.

We finally recover $A(t)(\underline{\rw}(t)) \in \fH_n(t)$ from $\dot{A}(t)(\dot{\rw}(t)) \in \fV_n(t)$ via suitable time integration. For this, we are assuming that the spaces $\fH_n(t)$ for the state and $\fV_n(t)$ for the velocity are coherent with each other with respect to time integration.
Note that, at any time $t$, the reconstruction relies on $\underline{\rw}(t)$, that is, on the observations in the interval $[0,t]$.

We can directly obtain an estimate on the recovery error
\begin{equation}\label{eq:rec-err}
  e(t) \coloneqq \norm{u^\dagger(t)-A(t)(\underline{\rw}(t))}_{\fH}\qquad\forall\, t>0
\end{equation}
that we provide in the next Lemma.

\begin{lemma}
  \label{lem:recovery-Vn-given}
  Assume that, for every $t\in\fT$, there exists a constant $\delta k(t)>0$ such that
  \begin{equation}\label{eq:near-opt}
    E_{\star}(\dot{A}(t), \dot{\cU}(t), \Vn(t)) \leq \delta k(t) \delta_{\star}(\dot{\cU}(t), \Vn(t)),\qquad \star=\{wc,\mu\}.
  \end{equation}
  Then, the recovery errors $e(t)$ in \eqref{eq:rec-err} and $E_{\star}(A(t), \cU(t), \Hn(t))$ can be bounded as
  \begin{align*}
    e(t)                            & \leq e(0) + \int_0^t \delta k(s) \min_{v\in \fV_n(s)}\norm{\dot{u}^\dagger(s)-v}_{\fH} \d s, \\
    E_{\star}(A(t), \cU(t), \Hn(t)) & \leq E_{\star}(A(0), \cU(0), \Hn(0))
    + \int_0^t \delta k(s) \delta_{\star}(\dot{\cU}(s),\Vn(s)) \d s, \quad \star \in \{wc, \mu\}.
  \end{align*}
\end{lemma}

\begin{proof}
  The result follows by standard time integration and differentiation.
\end{proof}

\subsection{Contributions and outline of the paper}
\label{sec:novelty-and-plan}

The strategy outlined above provides only a conceptual roadmap, since it assumes that the approximation space $\Vn(t)$ is known for every $t\in\fT$. In practice, this space must be constructed and evolved dynamically as new measurements become available. The main contribution of this paper is to develop a principled filtering strategy that jointly updates the approximation space $\Vn(t)$ and incorporates the available data, while preserving the geometric structure of the underlying dynamics. Since our focus is on Hamiltonian systems, we require the reconstructed trajectory to define a symplectic flow. This is achieved by imposing suitable geometric constraints on $\Vn(t)$. We complement our construction with a rigorous stability and error analysis of the resulting reconstruction. Because $\Vn(t)$ must itself be inferred during the evolution, the resulting recovery bounds are necessarily weaker than the idealized estimate from \Cref{lem:recovery-Vn-given}.

A central quantity in the analysis is the near-optimality constant $\delta k(t)$ in \eqref{eq:near-opt}. We will see later on that it  can be interpreted as a stability constant governed by the angle between the approximation space $\fV_n(t)$ and the observation space $\Wm$; see \cite{BCDDPW17, Mula2023}. Stable reconstruction therefore requires these spaces to remain sufficiently well aligned. Since Hamiltonian dynamics often involve transport effects, sharp gradients, and/or localized structures, a fixed observation space may fail to capture the main features of the dynamics, and this  translates into a loss of stability. We consequently allow the observation space to evolve as $\Wm(t)$, corresponding to sensors that move in time, following the optimal-control strategy introduced in \cite{MPV25}.

The outline of the paper is as follows. We close this introductory part with \Cref{sec:earlier-works} where we outline the main connections and novelties with respect to earlier works. We then move on to \Cref{sec:hamiltonian-dyn} where we recall relevant concepts and properties of Hamiltonian systems and symplectic flows. \Cref{sec:rec-strategy} introduces the filtering scheme and \Cref{sec:reformulation-decoders} provides a reformulation of the scheme in terms of decoder approximations. This provides a more abstract and general way of understanding the scheme, and it allows us to derive relevant theoretical properties. \Cref{sec:err} gives reconstruction error bounds and provides a stability analysis of the method. \Cref{sec:moving-sensors} explains the strategy to dynamically evolve the sensors' locations, and \Cref{sec:num-exp} gives numerical tests to illustrate the behavior and performance of the method. \Cref{sec:conclusion} summarizes the main conclusions and outlines possible future works. The text comes with two appendices: \Cref{app:proof-gamma-dyn} gives the proof of \Cref{thm:dynamics-gamma-2}, and \Cref{app:discrete} discusses practical aspects about discretization.

\subsection{Connections and novelties with respect to earlier works}
\label{sec:earlier-works}

Our approach lies at the intersection of dynamical approximation, data assimilation and filtering, and structure-preserving methods. We briefly review related work and position our contribution.

\begin{itemize}

  \item \textbf{Dynamical approximation.} Dynamically adapted approximation spaces have received increasing attention in forward PDE solvers, parametric model reduction, and uncertainty quantification, especially for moving or localized structures; see, e.g., \cite{MNV2020, P21, HPR22, SSBP2024, NT2024, BPV2024, FLLN2026, BCM2025, GMPR2026}. Their use in state estimation and filtering is much less explored and, to our knowledge, has focused mainly on linear Kalman filtering in a Bayesian setting; see \cite{Lombardi2022, SHNT2023, DY22,Vidlickova2022, NRT2026, KMNZ2026}.

  \item \textbf{Data assimilation, filtering, and Hamiltonian dynamics.} Bayesian data assimilation and filtering have been widely applied to Hamiltonian systems, particularly in weather applications; see, e.g., \cite{HH1998, WLNR2005, HRKR2021}. In contrast, our deterministic approach provides rigorous reconstruction error bounds and explicitly quantifies the role of the nature and placement of the measurements.

        Our framework is closely related to Parameterized Background Data-Weak (PBDW) methods, which combine reduced-order approximation and inverse-problem techniques. Introduced in \cite{MPPY15} and further developed in \cite{BCDDPW17, CDDFMN2020, CDMN2022, CDMS2022}, these methods mostly address stationary or smoothing problems with fixed approximation and observation spaces. Allowing both spaces to evolve extends this framework to filtering problems with transport, moving discontinuities, and localized structures.

        A time-dependent PBDW framework with evolving approximation and observation spaces was introduced in \cite{MPV25}. There, however, $\Vn(t)$ evolves solely according to the parametric forward dynamics, without feedback from incoming measurements. Here, we close this loop by dynamically adapting both $\Vn(t)$ and $\Wm(t)$ using the filtering information, with their evolution derived from variational principles. This coupled construction introduces new challenges in establishing stable reconstruction bounds while preserving the geometric structure of the Hamiltonian dynamics.
\end{itemize}

\section{Hamiltonian dynamics in Hilbert spaces}
\label{sec:hamiltonian-dyn}

\subsection{Definition}
\label{sec:hamitonian-flows}
A canonical Hamiltonian system in a Hilbert space $\fH$ is a triplet $(\fH, \omega, \cH)$ where  $\cH:\fH\to \bR$ is a so-called Hamiltonian function, which we assume to be differentiable, and $\omega$ is the canonical symplectic structure. In such a setting, the state variable $u\in \fH$ has two components, which we write as $u=(u^q,u^p)$. The variable $u^q$ represents the generalized position, and $u^p$ the conjugate momentum from Hamiltonian mechanics. Consequently, $\fH$ has the form of a product $\fH=\Hhat\times \Hhat$ with $\Hhat$ a Hilbert space with inner product $\inprod{\cdot}{\cdot}{\Hhat}$. We endow the space $\fH$ with its natural inner product
$$
  \inner{v,w}_{\fH} \coloneqq \inner{v^q,w^q}_{\Hhat}+\inner{v^p,w^p}_{\Hhat},
  \quad \forall v=(v^q,v^p),\; w=(w^q,w^p) \in \fH.
$$
The symplectic structure is defined as the bilinear form $\omega:\fH\times \fH\rightarrow\bR$ such that
\begin{equation*}
  \omega(u,v):=\inprodV{\cJ (u)}{v}, \quad \forall u,v\in \fH
\end{equation*}
and $\cJ:\fH\rightarrow \fH$ is the linear mapping
\begin{equation}\label{eq:J}
  \cJ \left(\prt*{v^q,v^p}\right)=(v^p,-v^q), \quad \forall v=(v^q,v^p) \in \fH.
\end{equation}
Note that $\cJ^2 = -\id$ and $\cJ$ is skew-adjoint with respect to the inner product of $\fH$, so that $\omega$ is skew-symmetric.

Given a canonical Hamiltonian system $(\fH, \omega, \cH)$, and an initial condition $u_0\in \fH$, the associated Hamiltonian dynamics is described by the differentiable curve $u:\fT\to \fH$ that satisfies
\begin{equation*}
  \begin{cases}
    \dot u(t) = \cJ\left(\grad_{\fH} \cH(u(t))\right)  \in \fH & \quad\forall t>0, \\
    u(0) = u_0\in\fH,
  \end{cases}
\end{equation*}
where, for any $v\in \fH$, $\grad_{\fH} \cH(v) \in \fH$ is the gradient of $\cH$ in $\fH$. It is defined as the Riesz representer of $(D\cH)_v\in \fH'$, the differential of $\cH$ at $v$, which is the unique element of $\fH$ such that
$$
  (D\cH)_v(w) = \inner{\grad_{\fH} \cH(v), w}_{\fH}, \quad \forall w\in \fH.
$$
Along solutions, the value of the Hamiltonian is preserved since, by the chain rule,
$$
  \frac{\d }{\d t} \cH(u(t))
  = (D\cH)_{u(t)}(\dot u (t))
  = \inner{\grad_{\fH} \cH(u), \dot u}_{\fH}
  = \omega(\grad_{\fH} \cH(u), \grad_{\fH} \cH(u))
  =0.
$$
Moreover, the flow $\Phi_t$ of a Hamiltonian system is a symplectic map, that is $\omega(\Phi_t(v_1),\Phi_t(v_2))=\omega(v_1,v_2)$ for all $v_1,v_2\in\fH$, see \cite[Section 18]{Cannas2001}.
To ensure that the approximate flow of a Hamiltonian problem is symplectic, one can enforce that the approximate velocity field  remains on the tangent space of a symplectic manifold. This is the idea behind our approach where we approximate the velocity $\dot{u}^{\dagger}$ in the tangent spaces of evolving symplectic vector spaces.

Many well-known PDEs can be expressed as Hamiltonian systems in Hilbert spaces. Some of the most common examples include wave models such as the Schrödinger equation and the Korteweg–de Vries equation, compressible and
incompressible Euler equations in fluid dynamics, Vlasov–Poisson and Vlasov–Maxwell equations in plasma physics.

As a concrete example, let us consider the Schrödinger equation on a Lipschitz domain $\fD\subseteq \bR^d$.
The problem is posed on $\fH=\Hhat\times\Hhat$ with $\Hhat=L^2(\fD)$, and the Hamiltonian function is defined as
\begin{equation*}
  \cH(u) =
  \begin{cases}
    \frac 1 2 \prt*{\sum_{i=1}^d\norm{\partial_{x_i} u}^2_{\fH}} - \frac \eps 4 \norm{u}_{\fH}^4, & \quad \forall u=(u^q,u^p) \in \dom(\cH)\coloneqq H^1_0(\fD)^2\cap L^4(\fD)^2, \\
    +\infty                                                                                       & \quad \forall  u=(u^q,u^p) \in \fH \setminus \dom(\cH).
  \end{cases}
\end{equation*}
Here the quantities $u^q$ and $u^p$ represent the real and imaginary part of a complex wave function. To define the associated Hamiltonian vector field, we compute the Fréchet subdifferential of $\cH$ at a given $u\in \dom(\cH)$, which is given for all $v\in \fH$ by
\begin{equation*}
  (D\cH^F)_u(v) =
  \begin{cases}
    \int_D \prt*{-\Delta u^q - \eps u^q\norm{u}_{\fH}^2 }v^q
    + \prt*{-\Delta u^p - \eps u^p\norm{u}_{\fH}^2 }v^p, & \quad \forall u \in \dom(D\cH^F)   ,            \\
    \emptyset,                                           & \quad \forall u \in \fH \setminus \dom(D\cH^F),
  \end{cases}
\end{equation*}
where
$\dom(D\cH^F) = H^2(\fD)^2\cap H^1_0(\fD)^2\cap L^6(\fD)^2$.
We can directly identify the gradient of $\cH$ for any $u \in \dom(D\cH^F)$ as
$$
  \grad_{\fH} \cH(u) =
  \begin{pmatrix}
    -\Delta u^q - \eps u^q\norm{u}_{\fH}^2 \\
    -\Delta u^p - \eps u^p\norm{u}_{\fH}^2
  \end{pmatrix}, \qquad \forall u \in \dom(D\cH^F).
$$
It follows that the strong form of the Hamiltonian dynamics of our example reads: given an initial condition $u_0\in \dom(D\cH^F)$, find $u=(u^q,u^p)$, such that for $t>0$,
\begin{equation*}
  \begin{cases}
    \partial_t {u}^q & = - \Delta u^p-\varepsilon ((u^q)^2+(u^p)^2)u^p, \\
    \partial_t {u}^p & = \Delta u^q+\varepsilon ((u^q)^2+(u^p)^2)u^q,
  \end{cases}
\end{equation*}
and $u(t=0)=u_0$.

To connect back with the filtering problem from \Cref{sec:filtering}, we need to consider that the Hamiltonian depends on parameters $\theta\in {\fTheta}$, so we write $\cH_\theta:\fH\to \bR$, and consider parametric Hamiltonian systems of the form
\begin{equation*}
  \begin{cases}
    \partial_t u = \cJ\left(\grad_{\fH} \cH_\theta(u)\right) & \quad (t,x)\in \fT\times \fD, \\
    u(\theta)(t=0, x) = u_0(\theta)(x)                       & \quad x\in\fD.
  \end{cases}
\end{equation*}
In the above example about the Schrödinger equation, $\eps$ could be seen as a parameter that ranges in a certain interval, and we could also consider a parametric family of initial conditions $u_0$. The family of parametric trajectories generates the set of solutions $\cU$ from \cref{eq:set-param-sols} for which we assume that assumptions (A1), (A2), (A3) hold.

\subsection{The manifold of orthosymplectic functions}
\label{sec:orthosymplectic}
To ensure that our reconstruction is a symplectic flow map, we approximate $u^\dagger(t)$ in spaces $\fH_{2n}(t)$ spanned by orthogonal symplectic bases. Exploiting the manifold structrure of such bases, we evolve them along the corresponding tangent spaces, which serve as the velocity approximation spaces from \Cref{sec:idealized-strategy} that we use to approximate $\dot{u}^\dagger(t)$.

The set of all $2n$-tuples of orthosymplectic functions is defined as
$$
  \cV_{2n} \coloneqq \{ \prt*{v_1,\dots, v_{2n}} \in \fH^{2n} \cond \inner{v_i,v_j}_{\fH}=\delta_{i,j} \text{ and } \omega(v_i,v_j)=(\bJ_{2n})_{i,j} \} \subset \fH^{2n},
$$
where
$$
  \bJ_{2n}\coloneqq \begin{bmatrix}
    0      & \bI_n \\
    -\bI_n & 0
  \end{bmatrix},
$$
and $\bI_n\in\bR^{n\times n}$ is the identity matrix.

For a given $V\in \cV_{2n}$, by considering a curve $Q:(-\eps,\eps)\to \cV_{2n}$ such that $Q(0)=V$, and differentiating in time the orthosymplectic constraint, we obtain that the tangent space of $\cV_{2n}$ at $V$ is given by
\begin{equation*}
  \tangent{V}{\cV_{2n}} = \{\prt*{\delta v_1,\dots, \delta v_{2n}} \in \fH^{2n} \cond \inner{\delta v_i,v_j}_{\fH}+\inner{v_i,\delta v_j}_{\fH}=0 \text{ and } \omega(\delta v_i,v_j)+\omega(v_i,\delta v_j)=0 \}.
\end{equation*}
To single out a unique parametrization of the elements of $\tangent{V}{\cV_{2n}}$, we impose gauge constraints and confine ourselves to
its horizontal component, defined as
\begin{equation}
  \label{eq:tangent-HV}
  \widetilde{\tangent{V}{\cV_{2n}}} \coloneqq \{\prt*{\delta v_1,\dots, \delta v_{2n}} \in \fH^{2n} \cond \inner{\delta v_i,v_j}_{\fH}=0 \text{ and } \delta v_i=\cJ(\delta v_{i+n})\;\forall 1\leq i\leq n \}
  \subset \tangent{V}{\cV_{2n}}.
\end{equation}
Note that, if $\delta V\in\widetilde{T_{V} \cV_{2n}}$, then $\delta v_i = \sum_{j=1}^{2n}(\bJ_{2n})_{i,j}\cJ(\delta v_j)$ for all $1\leq i\leq2n$.

In our scheme, we approximate $\dot{u}^\dagger(t)$ in a velocity space $\fV_{2n}(t)=\vspan\{ \delta V(t) \}$ where $\delta V(t)\in \widetilde{\tangent{V(t)}{\cV_{2n}}}$, and $V(t)\in \cV_{2n}$ will be the orthosymplectic basis spanning the approximation space $\fH_{2n}(t)=\vspan\{ V(t)\}$.

\subsection{Remarks on notation}
In what follows, we use the notation $$\inner{u,V}_{\fH}:=(\inner{u,v_i}_{\fH})_{i=1,\dots,2n}\in\bR^{2n}$$ and $$\inner{W,V}_{\fH}:=(\inner{w_i,v_i}_{\fH})_{i=1,\dots,2n}\in\bR^{2n}$$ for $u\in\fH$ and $V,W\in\fH^{2n}$. Moreover, we regard elements of $\fH^{2n}$ as column vectors in the following sense: for any scalar $\alpha\in\bR$ and for any matrix $\bA\in\bR^{s\times2n}$ we define $\alpha V:=(\alpha v_i)_{i=1,\dots,2n}\in\fH^{2n}$ and $\bA V:=(\sum_{j=1}^{2n}\bA_{i,j}v_j)_{i=1,\dots,s}\in\fH^s$.
Finally, for ease of notation we omit parentheses when applying the linear operators $\cJ$ and $\proj{\fV}$, for any subspace $\fV$ of $\fH$. We write $\cJ V:=(\cJ v_i)_{i=1,\dots,2n}\in\fH^{2n}$ for $V\in\fH^{2n}$, and $\proj{\fV}V:=(\proj{\fV}v_i)_{i=1,\dots,2n}\in\fH^{2n}$ for any $\fV\subset\fH$.

\section{The filtering scheme}
\label{sec:rec-strategy}
Our reconstruction algorithm follows the idealized approach introduced in \Cref{sec:idealized-strategy}, with the extra ingredient that we introduce a strategy to dynamically evolve the approximation space $\fH_{2n}(t)$:

\paragraph{At $t=0$:} We reconstruct $u^\dagger(0)$ as follows:
\begin{itemize}
  \item We build a linear approximation space $\fH_{2n}(0)=\vspan\{V(0)\} \subset \fH$ where $V(0)\in \cV_{2n}$.
        We postpone to \Cref{sec:filtering-scheme} the detailed description of how the	 orthosymplectic basis $V(0)$ is constructed.
  \item Given the measurements $\rw(0)$ and the approximation space $\fH_{2n}(0)$, we solve the least-squares problem
        \begin{equation}\label{eq:r0}
          r(0) = \argmin_{v\in \fH_{2n}(0)} \frac 1 2 \norm{\rw(0) - \proj{\Wm(0)}(v) }^2_{\fH}.
        \end{equation}
        This problem has a unique minimizer (see \cite{BCDDPW17,Mula2023}) provided
        $$
          \beta(\fH_{2n}(0), \Wm(0) ) >0,
        $$
        where, for any pair of subspaces $X, Y$ of $\fH$, $\beta(X,Y)$ is defined as
        \begin{equation*}
          \beta(X, Y) \coloneqq \inf_{x\in X} \sup_{y\in Y} \dfrac{\left< x, y \right>_{\fH}}{\norm{x}_{\fH}\norm{y}_{\fH}}
          \; \in [0, 1].
        \end{equation*}
        Moreover, it holds (see \cite{MPPY15, BCDDPW17,Mula2023})
        $$
          e(0)=\norm{u^\dagger(0)-A(0)(\rw(0))}_{\fH} \leq \beta^{-1}(\fH_{2n}(0), \Wm(0) ) \min_{v\in \fH_{2n}(0)} \norm{u^\dagger(0)-v}_{\fH}
        $$
        and
        $$
          E_{\star}(A(0),\Ucal(0),\fH_{2n}(0)) \leq \beta^{-1}(\fH_{2n}(0), \Wm(0) ) \delta_{\star}(\Ucal(0),\fH_{2n}(0))
          \quad \star=\{wc,\mu\}.$$
\end{itemize}
\paragraph{For $t>0$:} To approximate the trajectory $u^\dagger(t)$ by a symplectic flow, we seek a reconstruction of the form
\begin{equation*}
  r(t) = \sum_{i=1}^{2n} a_i(t)v_i(t),
\end{equation*}
where $a_i(t)\in \bR$, and $V(t)=\set{v_i(t)\in \fH}_{i=1}^{2n} \in \cV_{2n}$ is an orthosymplectic basis, so that $r(t) \in \fH_{2n}(t)=\vspan\{ V(t) \}$. Its velocity, which approximates $\dot{u}^\dagger(t)$, is
\begin{equation}\label{eq:rdot}
  \dot{r}(t) = \sum_{i=1}^{2n} \prt{\dot{a}_i(t) v_i(t)+a_i(t) \dot{v}_i(t)},
\end{equation}
where $\dot{V}(t) = \prt*{\dot{v}_1,\dots, \dot{v}_{2n}}\in T_{V(t)} \cV_{2n}$. At first glance, one would like to proceed similarly as for $t=0$, and solve
\begin{equation*}
  r(t) \in \argmin_{v \in \fH_{2n}(t)} \frac 1 2 \norm{ \rw(t) - \proj{\Wm}(v)}^2_{\fH}.
\end{equation*}
However, this will not provide a strategy to evolve $\fH_{2n}(t)$ over time. For this, one must work at the level of the velocity, and the first attempt would be to consider
\begin{equation}
  \label{eq:rec-vel-1}
  (\dot{a}(t), \dot{V}(t)  ) \in \argmin_{(\delta a, \delta V) \in \bR^{2n}\times  T_{V(t)} \cV_{2n}}
  \frac 1 2 \Norm{ \dot{\rw}(t) -  \proj{\Wm}\Big(\sum_{i=1}^{2n}\prt{\delta a_i v_i(t) + a_i(t)\delta v_i}\Big)}^2_{\fH}.
\end{equation}
We would then recover $\dot{r}(t)$ as in \eqref{eq:rdot},
and obtain $r(t)$ by time integration.  Unfortunately, the necessary conditions connected to \eqref{eq:rec-vel-1} are not enough to prescribe a well-defined evolution for the orthosymplectic basis $V(t)$. We only obtain an evolution of $V(t)$ on certain subspaces of $\fH$, and not in the whole ambient space.
To overcome this issue, we search for a space $\fH_{2n}(t)$ which not only approximates well $u^\dagger(t)=u(\theta^\dagger)(t,\cdot)$, but also $u(\theta)(t,\cdot)$ for all $\theta\in {\fTheta}$. Thus, for all $\theta\in {\fTheta}$, we approximate $u(\theta)(t,\cdot)$ with
$$
  s(\theta)(t,\cdot) = \sum_{i=1}^{2n} c_i(t)(\theta) v_i(t),
$$
and $\dot{u}(\theta)(t,\cdot)$ with
$$
  \dot{s}(\theta)(t,\cdot) = \sum_{i=1}^{2n} \prt{\dot{c}_i(t)(\theta) v_i(t)+c_i(t)(\theta) \dot{v}_i(t)}.
$$
Consequently, we search for
\begin{align*}
  (\dot{a}(t), \dot{c}(t), \dot{V}(t)  )
  \in
   & \argmin_{ \substack{(\delta a, \delta V) \in \bR^{2n}\times  \widetilde{T_{V(t)} \cV_{2n}}                        \\ \delta c(\theta) \in \bR^{2n},\,\forall \theta\in {\fTheta} }}
  \frac12 \Norm{\dot{\rw}(t) -  \proj{\Wm}\Big(\sum_{i=1}^{2n}\prt{\delta a_i v_i(t) + a_i(t)\delta v_i}\Big)}^2_{\fH} \\
   & + \frac \lambda 2 \int_{\fTheta}
  \Norm{\cJ\prt{\grad_{\fH} \cH_\theta(s(\theta)(t,\cdot))} - \sum_{i=1}^{2n}\prt{\delta c_i(\theta) v_i(t) + c_i(t)(\theta)\delta v_i}}^2_{\fH} \mu(\d \theta),
\end{align*}
where $\lambda>0$ is a regularization parameter, and where $\dot{V}(t)$ is searched in the horizontal component $\widetilde{T_{V(t)} \cV_{2n}}$ \eqref{eq:tangent-HV} of the tangent space $T_{V(t)} \cV_{2n}$. This is done to ensure uniqueness of the minimizer as we show later on in \Cref{thm:dynamics-gamma}.

\section{Reformulation with symplectic decoders}
\label{sec:reformulation-decoders}
We next express the above approximation strategy in the language of decoder mappings. This will allow us to build a higher conceptual abstraction of our procedure, and it will also ease certain technical developments related to the analysis of the properties of the scheme.

From \Cref{sec:rec-strategy}, our strategy for filtering is based on a dynamical approximation of both:
\begin{itemize}
  \item the unknown function $u^\dagger(t)\in \fH$, and
  \item the family of parametrized Hamiltonian trajectories $u(\theta)(t,\cdot)$ for $\theta\in {\fTheta}$ from \eqref{eq:solution-set-t}, which is a subset of $\fL$.
\end{itemize}
Consequently, we define a decoder mapping of the form
\begin{align*}
  \varphi:\fGamma & \to \fH \times \fL                                             \\
  \gamma          & \mapsto \varphi(\gamma)=(\varphi_1(\gamma), \varphi_2(\gamma))
\end{align*}
where the parameter space of the decoder is defined as the set of all elements $\gamma=(a, c, V)$ in
$$
  \fGamma \coloneqq  \bR^{2n} \times L^2(({\fTheta},\mu); \bR^{2n})\times \cV_{2n}.
$$
For every $\gamma =(a,c,V) \in \fGamma$, the decoder mapping $\varphi$ is defined as
\begin{equation*}
  \varphi(\gamma) \coloneqq (\varphi_1(\gamma), \varphi_2(\gamma)),
  \quad
  \text{where}
  \begin{cases}
    \varphi_1(\gamma) & = \sum_{i=1}^{2n} a_i v_i, \\
    \varphi_2(\gamma) & = \sum_{i=1}^{2n} c_i v_i. \\
  \end{cases}
\end{equation*}
These objects depend on given $x\in \fD$ and $\theta\in {\fTheta}$ as
\begin{equation*}
  \begin{cases}
    \varphi_1(\gamma)(x)         & = \sum_{i=1}^{2n} a_i v_i(x),         \\
    \varphi_2(\gamma)(\theta)(x) & = \sum_{i=1}^{2n} c_i(\theta) v_i(x). \\
  \end{cases}
\end{equation*}
In our filtering scheme, we want to build a time-dependent curve $\gamma:\fT\to \fGamma$ such that
\begin{align*}
  \begin{cases}
    u^\dagger(t)(x) \approx \varphi_1(\gamma(t))(x)        & = \sum_{i=1}^{2n} a_i(t) v_i(t)(x),         \\
    u(\theta)(t, x)\approx \varphi_2(\gamma(t))(\theta)(x) & = \sum_{i=1}^{2n} c_i(t)(\theta) v_i(t)(x). \\
  \end{cases}
\end{align*}
We may observe that $\varphi_1(\gamma(t))$ plays the role of $r(t)$ in our previous description of the scheme of \cref{sec:rec-strategy}, and $\varphi_2(\gamma(t))$ represents $s(\theta)(t,\cdot)$.

We next discuss some properties of $\varphi$ that are relevant to present the final formulation of our filtering scheme with the decoder, as well as to analyze its properties.

\subsection{Properties of the decoder}
\label{sec:decoder-def}

\paragraph{Differentiability of $\varphi$ and tangent spaces.}
$\fGamma$ is a manifold embedded in the space
$$
  \fG \coloneqq \bR^{2n} \times L^2(({\fTheta},\mu); \bR^{2n}) \times \fH^{2n}
$$ with inner product
$$
  \inner{{\gamma}_1, {\gamma}_2 }_{\fG}
  \coloneqq
  \inner{{a}_1, {a}_2 }_{\bR^{2n}}
  +
  \inner{{c}_1, {c}_2 }_{L^2(({\fTheta},\mu); \bR^{2n})}
  +
  \sum_{i=1}^{2n} \inner{({v}_{1})_i, ({v}_{2})_i }_{\fH}, \quad \forall \gamma_1,\gamma_2\in \fGamma.
$$
For every $\gamma=(a,c,V)\in\fGamma$, the tangent space of $\fGamma$ at a point $\gamma =(a,c,V)$ is given by
$$
  \tangent{\gamma}{\fGamma} = \bR^{2n} \times L^2(({\fTheta},\mu); \bR^{2n}) \times \tangent{V}{\cV_{2n}} \; \subset \fG.
$$
The following Lemma shows that $\varphi$ is differentiable.

\begin{lemma}
  \label{lem:differentiability}
  For every $\gamma\in \fGamma$, $\varphi$ is differentiable, and the differential map $(D\varphi)_\gamma \in \Lin(\tangent{\gamma}{\fGamma}; \fH\times \fL)$ is such that
  \begin{equation}\label{eq:diff_map}
    (D\varphi)_{\gamma}(\delta\gamma)
    = \Big(\sum_{i=1}^{2n}  \big(\delta a_iv_i + a_i \delta v_i\big), \sum_{i=1}^{2n}  \big(\delta c_iv_i + c_i \delta v_i\big)\Big),
    \quad \forall \delta\gamma=(\delta a, \delta c, \delta V) \in \tangent{\gamma}{\fGamma}.
  \end{equation}
\end{lemma}
\begin{proof}
  Since $\varphi(\gamma)$ is a sum and a product of the input coordinates of $\gamma$, $\varphi$ is differentiable, and the differential at a point $\gamma\in\fGamma$ is a linear mapping $(D\varphi)_\gamma \in \Lin(\tangent{\gamma}{\fGamma}; \fH\times\fL)$ which maps elements from $\tangent{\gamma}{\fGamma}$ to $\fH\times\fL$. Let $\gamma\in \fGamma$, and consider a differentiable curve $\Gamma:(-\eps,\eps)\to \fGamma$ with $\Gamma(0)=\gamma=(a,c,V)\in \fGamma$ and $\Gamma'(0)=\prt{\dot{a}, \dot{c}, \dot{V}}\in \tangent{\gamma}{\fGamma}$. Since $\varphi$ is differentiable, by the chain rule,
  $$
    \left.\frac{\d (\varphi\circ \Gamma)}{\dt}(t)\right\vert_{t=0}
    = (D\varphi)_{\gamma}(\Gamma'(0))
    = \Big(\sum_{i=1}^{2n} \big(\dot{a_i}v_i + a_i \dot{v}_i\big), \sum_{i=1}^{2n}  \big(\dot{c_i}v_i + c_i \dot{v}_i\big)\Big).
  $$
\end{proof}

The following Lemma is a technical result that expresses that $(D\varphi)_\gamma$ has a certain adjoint property. We explain in \Cref{rem:adjoint-differential} how exactly this property connects to the formal adjoint of $(D\varphi)_\gamma$.

\begin{lemma}
  \label{lem:adjoint-like}
  For every $\gamma=(a,c,V)\in \fGamma$,
  $$
    \inner{(D\varphi)_{\gamma}(\delta \gamma), (h,\ell)}_{\fH\times\fL}
    =
    \inner{\delta\gamma, T(h, \ell)}_{\fG},
    \quad \forall \delta\gamma\in \tangent{\gamma}{\fGamma},\; (h,\ell)\in \fH\times\fL,
  $$
  where
  \begin{equation}\label{eq:Tadj}
    \begin{aligned}
      T: \fH\times\fL & \to \fG                                   \\
      (h,\ell)        & \mapsto T(h,\ell)=(\inner{h,V}_{\fH},f,g)
    \end{aligned}
  \end{equation}
  and
  \begin{align*}
     & f(\theta) \coloneqq \inner{\ell(\theta), V}_{\fH},
    \quad \forall \theta \in {\fTheta} \; \text{$\mu$-a.e.}                                                                                    \\
     & g\in\fH^{2n},\quad g_i \coloneqq a_i h + \int_{{\fTheta}} c_i(\theta)\ell(\theta) \mu(\d \theta) \in \fH,\quad \forall\,1\leq i\leq 2n.
  \end{align*}
\end{lemma}

\begin{proof}
  The definition of the differential map $(D\varphi)_{\gamma}$ in \eqref{eq:diff_map} yields, for a generic $\delta\gamma=(\delta a,\delta c, \delta V)\in \tangent{\gamma}{\fGamma}$ and $(h,\ell)\in\fH\times\fL$,
  \begin{equation}\label{eq:inprod_diff_map}
    \inner{(D\varphi)_{\gamma}(\delta\gamma), (h,\ell)}_{\fH\times\fL} = \inner{\delta a,\inner{h,V}_{\fH}}_{\bR^{2n}} + \sum_{i=1}^{2n}\inner{\delta c_iv_i, \ell}_{\fL} +\sum_{i=1}^{2n}\left(a_i\inner{\delta v_i,h}_{\fH} + \inner{c_i\delta v_i,\ell}_{\fL}\right).
  \end{equation}
  The second term in \eqref{eq:inprod_diff_map} gives
  \begin{equation*}
    \sum_{i=1}^{2n}\inner{\delta c_iv_i, \ell}_{\fL} = \sum_{i=1}^{2n}\int_{\fTheta} \delta c_i(\theta)\inner{v_i,\ell(\theta)}_{\fH}\,\mu(\d\theta) = \int_{\fTheta}\inner{\delta c(\theta),\inner{\ell(\theta),V}_{\fH}}_{\bR^{2n}}\,\mu(\d\theta) = \inner{\delta c,f}_{L^2(({\fTheta},\mu);\bR^{2n})}.
  \end{equation*}
  The last term in \eqref{eq:inprod_diff_map} can be written as
  \begin{equation*}
    \sum_{i=1}^{2n}\left(\inner{\delta v_i,a_ih}_{\fH}+\int_{\fTheta}c_i(\theta)\inner{\delta v_i,\ell(\theta)}_{\fH}\mu(\d\theta)\right)
    = \sum_{i=1}^{2n}\left\langle\delta v_i, a_ih+\int_{\fTheta} c_i(\theta)\ell(\theta)\,\mu(\d\theta)\right\rangle_{\fH} = \inner{\delta v, g}_{{\fH}^{2n}}.
  \end{equation*}
\end{proof}

\begin{remark}
  \label{rem:adjoint-differential}
  The adjoint of $(D\varphi)_{\gamma}:\tangent{\gamma}{\fGamma} \to \fH\times\fL $ is the map $(D\varphi)^*_\gamma : \dom((D\varphi)^*_\gamma) \to \tangent{\gamma}{\fGamma} $, where
  $$
    \dom((D\varphi)^*_\gamma)
    \coloneqq \{ (h,\ell)\in \fH\times\fL\cond  \exists \delta \eta \in \tangent{\gamma}{\fGamma} \st
    \inner{(D\varphi)_{\gamma}(\delta \gamma), (h,\ell)}_{\fH\times\fL}
    =
    \inner{\delta \gamma, \delta \eta}_{\fG},
    \; \forall \delta \gamma\in \tangent{\gamma}{\fGamma}
    \}.
  $$
  The mapping $T$ from \eqref{eq:Tadj} is related to the adjoint of $(D\varphi)_\gamma$ in the sense that $(D\varphi)^*_\gamma = T \vert_{\dom((D\varphi)^*_\gamma)}$.
\end{remark}

\paragraph{Injectivity of the differential map.} We next show that the differential $(D\varphi)_\gamma$ is injective under certain restrictions on $\gamma$, and on the input space $T_\gamma \fGamma$ where $(D\varphi)_\gamma$ is defined. This property will be crucial to determining the parameter dynamics of our filtering scheme.

Let $\bC:L^2(({\fTheta},\mu); \bR^{2n})\to \bR^{2n\times 2n}$ be the matrix-valued function defined as
\begin{equation}\label{eq:matrixC}
  (\bC(c))_{i,j}:=
  \int_{{\fTheta}} c_i(\theta)c_j(\theta)\, \mu(\d\theta),\qquad\forall\, 1\leq i,j\leq 2n, \quad \forall c\in L^2((\fTheta,\mu);\bR^{2n}).
\end{equation}
We next introduce the following subset of $L^2(({\fTheta},\mu); \bR^{2n})$
\begin{equation}\label{eq:space-C2n}
  \cC_{2n}\coloneqq \{c\in L^2(({\fTheta},\mu); \bR^{2n})\cond \mathrm{rank}(\bS(c))=2n\},
\end{equation}
where
\begin{equation}
  \label{eq:S-matrix}
  \bS(c) \coloneqq \bC(c)+\bJ_{2n}^{\top}\bC(c)\bJ_{2n}
\end{equation}
for $\bC(c)$ defined in \eqref{eq:matrixC}. The matrix $\bS(c)$ is symmetric and positive semidefinite. Indeed, for every $z\in\bR^{2n}$,
\begin{equation*}
  z^\top\bS(c)z
  =\int_{\fTheta}\left(\prt*{z^\top c(\theta)}^2+\prt*{(\bJ_{2n}z)^\top c(\theta)}^2\right)\mu(\d\theta)\geq0.
\end{equation*}
Consequently, the full-rank condition in \eqref{eq:space-C2n} is equivalent to the positive definiteness of $\bS(c)$. We define
$$
  \widetilde{\fGamma} \coloneqq \bR^{2n} \times \cC_{2n} \times \cV_{2n} \subset \fGamma.
$$
Finally, we also introduce
$$
  \widetilde{\tangent{\gamma}{\fGamma}} = \bR^{2n} \times L^2(({\fTheta},\mu); \bR^{2n}) \times \widetilde{\tangent{V}{\cV_{2n}}} \;
  \subset \tangent{\gamma}{\fGamma} \;
  \subset \fG .
$$
Based on these two subsets, we have the following injectivity result.

\begin{lemma}
  \label{lem:injective-decoder}
  For every $\gamma \in \widetilde{\fGamma}$, the restriction of $(D\varphi)_{\gamma}$ to $\widetilde{\tangent{\gamma}{\fGamma}}$,
  \begin{align*}
    (D\varphi)_{\gamma} \vert_{\widetilde{\tangent{\gamma}{\fGamma}}}: \; \widetilde{\tangent{\gamma}{\fGamma}} & \longrightarrow \fH\times \fL                                                                                         \\
    (\delta a, \delta c, \delta V)                                                                              & \longmapsto \bigg(\sum_{i=1}^{2n} (\delta a_i v_i+a_i\delta v_i), \sum_{i=1}^{2n}(\delta c_i v_i+c_i\delta v_i)\bigg)
  \end{align*}
  is an injective mapping.
\end{lemma}
\begin{proof}
  Let $\gamma=(a,c,V)\in \widetilde{\fGamma}$, and let $\delta\gamma=(\delta a, \delta c, \delta V)\in \widetilde{\tangent{\gamma}{\fGamma}}$ be such that $(D\varphi)_{\gamma}(\delta\gamma) = 0$. This is equivalent to
  \begin{equation*}
    \sum_{i=1}^{2n}\big(\delta a_iv_i+a_i\delta v_i\big) = 0 \quad\text{ and }\quad
    \sum_{i=1}^{2n}\big(\delta c_i v_i+c_i\delta v_i\big)= 0.
  \end{equation*}
  We next show that this implies that $\delta \gamma =0$, which proves injectivity of  $(D\varphi)_{\gamma} \vert_{\widetilde{\tangent{\gamma}{\fGamma}}}$.

  Since $\inner{\delta v_i, v_j}_{\fH}=0$ and $\inner{v_i,v_j}_{\fH}=\delta_{i,j}$ for all $1\leq i, j\leq 2n$, it follows that $\delta a_i=\delta c_i=0$ for all $1\leq i\leq 2n$.
  This implies
  \begin{equation}
    \label{eq:varphi-injective-1}
    \sum_{i=1}^{2n} a_i\delta v_i = 0 \quad\text{ and }\quad
    \sum_{i=1}^{2n} c_i\delta v_i= 0.
  \end{equation}
  We next show that $\delta V=0$. Multiplying by $c_j$ the second equality in \eqref{eq:varphi-injective-1} and integrating over $\fTheta$, we obtain
  \begin{equation}
    \label{eq:E-cond}
    0 = \sum_{j=1}^{2n} \bC_{i,j}(c) \delta v_j, \quad 1\leq i\leq 2n
    \quad \Longleftrightarrow \quad
    \bC(c) \delta V = 0\in\fH^{2n}.
  \end{equation}
  We next apply the operator $\cJ$ to the above equation \eqref{eq:E-cond} and use the fact that $\delta V\in\widetilde{\tangent{V}{\cV_{2n}}}$.
  This gives
  \begin{equation}\label{eq:S-cond2}
    \bC(c)\bJ_{2n}^\top \delta V = 0 \quad \Longleftrightarrow \quad
    \bJ_{2n}^{\top}\bC(c)\bJ_{2n} \delta V = 0.
  \end{equation}
  Adding \Cref{eq:E-cond} and \Cref{eq:S-cond2} yields
  $$
    \bS(c)\delta V = 0,
  $$
  where $\bS(c)$ was defined in \eqref{eq:S-matrix}.
  Since $c\in \cC_{2n}$, the matrix $\bS(c)$ is full rank, thus invertible. Therefore $\delta V=0$, and this concludes the proof of  injectivity of $(D\varphi)_{\gamma} \vert_{\widetilde{\tangent{\gamma}{\fGamma}}}$.
\end{proof}

\subsection{Back to the filtering scheme: derivation of a parameter curve}
\label{sec:filtering-scheme}
Our goal is to find a well-chosen initial parameter $\gamma(0)$, and then to build a smooth curve $\gamma : \fT \to \fGamma $ such that, for every time $t>0$,
$$
  \frac{\d}{\dt} \varphi_1(\gamma(t)) \approx \dot{u}^\dagger(t),
  \quad \text{and}\quad
  \frac{\d}{\dt} \varphi_2(\gamma(t))(\theta) \approx \dot{u}(\theta)(t,\cdot), \quad \forall \theta \in {\fTheta}.
$$
Assuming that we have a smooth curve $\gamma\in \cC^1(\fT; \fGamma)$, we obtain by the chain rule and \Cref{lem:differentiability}
\begin{equation*}
  \frac{\d}{\dt} \varphi(\gamma(t)) = \prt*{D\varphi}_{\gamma(t)} \big(\dot{\gamma}(t)\big)
  =\prt*{\sum_{i=1}^{2n}  (\dot{a_i}v_i + a_i \dot{v}_i), \sum_{i=1}^{2n}  (\dot{c_i}v_i + c_i \dot{v}_i)}.
\end{equation*}
This translates into searching for
\begin{equation}\label{eq:dynamics-gamma}
  \dot{\gamma}(t)=(\dot{a}(t), \dot{c}(t), \dot{V}(t)  )
  \in \argmin_{ \delta \gamma \in  \widetilde{\tangent{\gamma(t)}{\fGamma}}} \cL_{\gamma(t)}(\delta \gamma), \quad \forall t>0
\end{equation}
where $\cL_{\gamma(t)}:\widetilde{\tangent{\gamma(t)}{\fGamma}} \to \bR$ is defined for all $\delta \gamma = (\delta a, \delta c, \delta V)\in \widetilde{\tangent{\gamma(t)}{\fGamma}}$ as
\begin{equation}\label{eq:loss}
  \begin{aligned}
    \cL_{\gamma(t)}(\delta \gamma)
    \coloneqq & \, \frac 1 2 \Norm{ \dot{\rw}(t) - \proj{\Wm}\Big(\sum_{i=1}^{2n}\prt{\delta a_i v_i(t) + a_i(t)\delta v_i}\Big)}^2_{\fH}                                                                        \\
              & + \frac \lambda 2 \int_{\fTheta} \Norm{\cJ\prt{\grad_{\fH} \cH_\theta(\varphi_2(\gamma(t))} - \sum_{i=1}^{2n}\prt{\delta c_i(\theta) v_i(t) + c_i(t)(\theta)\delta v_i}}^2_{\fH} \mu(\d \theta).
  \end{aligned}
\end{equation}
Problem \eqref{eq:dynamics-gamma} yields a system of evolution equations for the parameters $\gamma(t)$, whose explicit form is given in \Cref{thm:dynamics-gamma}. The evolution is expressed in terms of a variable $\chi(\gamma)\in(\fH_{2n}^\perp)^{2n}$ as
\begin{equation}\label{eq:chi}
  \chi_i(\gamma)\coloneqq a_i \proj{\fH_{2n}^\perp}\left(\dot{\rw}-\proj{\Wm}\dot{\varphi}_1(\gamma)\right) + \lambda \proj{\fH_{2n}^\perp}\int_{{\fTheta}} c_i(\theta) \cJ\prt{\grad_{\fH} \cH_\theta(\varphi_2(\gamma(t))} \mu(\d \theta),
\end{equation}
defined for every $\gamma=(a,c,V)\in\widetilde{\fGamma}$ and $1\leq i\leq2n$. It also involves the shorthand notation
\begin{equation}\label{eq:GVPWV}
  \bG(V, \proj{\Wm} V)=\prt{\inner{v_i, \proj{\Wm} v_j}_{\fH} }_{1\leq i, j\leq 2n}.
\end{equation}

\begin{theorem}
  \label{thm:dynamics-gamma}
  If $\beta(\fH_{2n}(t), \Wm(t))>0$ for all $t>0$, then the minimization problem \eqref{eq:dynamics-gamma} admits a unique solution
  $(\dot{a},\dot{c},\dot{V})\in\widetilde{\tangent{\gamma}{\fGamma}}$
  given by
  \begin{subequations}
    \label{eq:dynamics-dot}
    \begin{align}
      \dot{a}(t)         & = \bG(V(t), \proj{\Wm} V(t))^{-1} \prt*{\inner{\dot{\rw}(t), V(t)}_{\fH} - \bG(V(t), \proj{\Wm}\dot{V}(t)) a(t)}, \label{eq:dynamics-dot-a}                 \\
      \dot{c}(t)(\theta) & = \inner{\cJ\prt{\grad_{\fH} \cH_\theta(\varphi_2(\gamma(t))}, V(t)}_{\fH}, \quad \forall \theta \in{\fTheta} \;\text{$\mu$-a.e.},\label{eq:dynamics-dot-c} \\
      \dot{V}(t)         & =\lambda^{-1} \bS(c)^{-1}(\chi(\gamma(t))+ \bJ_{2n}\cJ \chi(\gamma(t))), \label{eq:dynamics-dot-V}
    \end{align}
  \end{subequations}
  where $\bS(c)$, $\bG(V(t), \proj{\Wm} V(t))$, and $\chi$ are defined in \eqref{eq:S-matrix}, \eqref{eq:GVPWV}, and \eqref{eq:chi}, respectively. Equations \eqref{eq:dynamics-dot} supplemented with the initial condition $\gamma(0)$ form a system of equations for the evolution of the parameters.
\end{theorem}

\begin{proof}
  Let $t>0$ and $\gamma\in\widetilde{\fGamma}$, and omit the dependence on time. We first prove that the minimization problem \eqref{eq:dynamics-gamma} admits a unique solution. To this end, we define the bounded linear observation operator
  \begin{align*}
    \cO:\fH\times\fL & \to \fH\times\fL,                             \\
    (x,y)            & \mapsto \prt*{\proj{\Wm}x,\sqrt{\lambda}\,y},
  \end{align*}
  and the observed decoder $\varphi^{\rm obs}\coloneqq\cO\circ\varphi$. Since $\cO$ is linear,
  $$
    (D\varphi^{\rm obs})_\gamma=\cO\circ(D\varphi)_\gamma.
  $$
  We denote the restriction of this differential to the admissible tangent space by
  \begin{align*}
    \cA_\gamma
    \coloneqq (D\varphi^{\rm obs})_\gamma\vert_{\widetilde{\tangent{\gamma}{\fGamma}}}:
    \widetilde{\tangent{\gamma}{\fGamma}} & \longrightarrow \fH\times\fL, \\
    (\delta a,\delta c,\delta V)
                                          & \longmapsto
    \prt*{
      \proj{\Wm}\sum_{i=1}^{2n}\prt{\delta a_i v_i+a_i\delta v_i},
      \sqrt{\lambda}\sum_{i=1}^{2n}\prt{\delta c_i v_i+c_i\delta v_i}
    }.
  \end{align*}
  With
  $$
    b_\gamma
    \coloneqq\prt*{
      \dot{\rw},
      \sqrt{\lambda}\,\cJ\prt{\grad_{\fH}\cH_\theta(\varphi_2(\gamma))}
    }\in\fH\times\fL,
  $$
  the loss in \eqref{eq:loss} can be written as
  \begin{equation}\label{eq:loss-operator}
    \cL_\gamma(\delta\gamma)
    =\frac 1 2 \norm{b_\gamma-\cA_\gamma(\delta\gamma)}_{\fH\times\fL}^2.
  \end{equation}

  We next establish the properties that guarantee existence and uniqueness of the minimizer of $\cL_\gamma$. First, for fixed $\gamma$, the admissible space $\widetilde{\tangent{\gamma}{\fGamma}}$ is a closed linear subspace of $\fG$ (we omit the proof of this statement). We next show that the linear operator $\cA_\gamma$ is injective. For every $\delta\gamma=(\delta a,\delta c,\delta V)\in\widetilde{\tangent{\gamma}{\fGamma}}$, the orthogonality of $\delta V$ to $V$ and the definition of $\bS(c)$ yield
  \begin{equation}
    \int_{\fTheta}\norm*{\sum_{i=1}^{2n}\prt{\delta c_i(\theta)v_i+c_i(\theta)\delta v_i}}_{\fH}^2\mu(\d\theta)=\norm{\delta c}_{L^2((\fTheta,\mu);\bR^{2n})}^2
    +\frac12\inner{\bS(c)\delta V,\delta V}_{\fH^{2n}}.
    \label{eq:coercivity-second-component}
  \end{equation}
  Since $c\in\cC_{2n}$, the matrix $\bS(c)$ is positive definite. Setting
  $$
    \kappa_c
    \coloneqq\sqrt{\lambda\min\left\{1,\frac12\lambda_{\min}(\bS(c))\right\}}>0,
  $$
  equation \eqref{eq:coercivity-second-component} and the definition of $\cA_\gamma$ imply
  \begin{equation}\label{eq:control-c-V}
    \prt*{\norm{\delta c}_{L^2((\fTheta,\mu);\bR^{2n})}^2
    +\norm{\delta V}_{\fH^{2n}}^2}^{1/2}
    \leq\kappa_c^{-1}\norm{\cA_\gamma(\delta\gamma)}_{\fH\times\fL}.
  \end{equation}
  Furthermore
  \begin{align*}
    \beta(\fH_{2n},\Wm)\norm{\delta a}_{\bR^{2n}}
     & \leq\norm*{\proj{\Wm}\sum_{i=1}^{2n}\delta a_i v_i}_{\fH}                     \\
     & \leq\norm*{\proj{\Wm}\sum_{i=1}^{2n}\prt{\delta a_i v_i+a_i\delta v_i}}_{\fH}
    +\norm{a}_{\bR^{2n}}\norm{\delta V}_{\fH^{2n}},
  \end{align*}
  and combining this estimate with \eqref{eq:control-c-V} gives
  $$
    \norm{\delta a}_{\bR^{2n}}
    \leq\beta(\fH_{2n},\Wm)^{-1}
    \prt*{1+\norm{a}_{\bR^{2n}}\kappa_c^{-1}}
    \norm{\cA_\gamma(\delta\gamma)}_{\fH\times\fL}.
  $$
  Consequently, there exists $\alpha_\gamma:=\kappa_c\left(1+\left(\frac{\kappa_c+\norm{a}_{\Rbb^{2n}}}{\beta(\fH_{2n},\Wm)}\right)^2\right)^{-1/2}>0$ such that
  \begin{equation}\label{eq:observed-decoder-lower-bound}
    \norm{\cA_\gamma(\delta\gamma)}_{\fH\times\fL}
    \geq\alpha_\gamma\norm{\delta\gamma}_{\fG},
    \qquad
    \forall\delta\gamma\in\widetilde{\tangent{\gamma}{\fGamma}}.
  \end{equation}
  In particular, $\cA_\gamma$ is injective, consistently with \Cref{lem:injective-decoder}: the vanishing of the second component of $\cA_\gamma(\delta\gamma)$ implies $\delta c=0$ and $\delta V=0$ by the argument in the proof of that lemma, while the vanishing of the first component and  the fact that  $\beta(\fH_{2n},\Wm)>0$ imply $\delta a=0$.

  We can now prove strong convexity and coercivity of $\cL_\gamma$. Since \eqref{eq:loss-operator} is a quadratic least-squares functional, its second differential satisfies
  $$
    (D^2\cL_\gamma)_{\delta\gamma}(\eta,\eta)
    =\norm{\cA_\gamma(\eta)}_{\fH\times\fL}^2
    \geq\alpha_\gamma^2\norm{\eta}_{\fG}^2,
    \qquad
    \forall\delta\gamma,\eta\in\widetilde{\tangent{\gamma}{\fGamma}},
  $$
  where we used \eqref{eq:observed-decoder-lower-bound}. Thus $\cL_\gamma$ is strongly convex on $\widetilde{\tangent{\gamma}{\fGamma}}$, which implies that it has at most one minimizer. Furthermore, \eqref{eq:loss-operator}, \eqref{eq:observed-decoder-lower-bound}, and Young's inequality give
  $$
    \cL_\gamma(\delta\gamma)
    \geq\frac14\norm{\cA_\gamma(\delta\gamma)}_{\fH\times\fL}^2
    -\frac12\norm{b_\gamma}_{\fH\times\fL}^2
    \geq\frac{\alpha_\gamma^2}{4}\norm{\delta\gamma}_{\fG}^2
    -\frac12\norm{b_\gamma}_{\fH\times\fL}^2.
  $$
  Therefore $\cL_\gamma$ is coercive. If $(\delta\gamma_k)_{k\geq1}$ is a minimizing sequence, coercivity implies that it is bounded in $\fG$. Since $\widetilde{\tangent{\gamma}{\fGamma}}$ is a closed linear subspace of the Hilbert space $\fG$, a subsequence converges weakly to some $\delta\gamma_\star\in\widetilde{\tangent{\gamma}{\fGamma}}$. The functional $\cL_\gamma$ is weakly lower semicontinuous, so $\delta\gamma_\star$ is a minimizer. Strong convexity shows that this minimizer is unique. Hence the minimization problem \eqref{eq:dynamics-gamma} has a unique solution.

  Since $\cL_{\gamma}$ is differentiable, the unique minimizer $\dot{\gamma}\in\widetilde{\tangent{\gamma}{\fGamma}}$ satisfies $(D\cL_\gamma)_{\dot{\gamma}}(\delta \gamma)=0$ for all $\delta \gamma\in \widetilde{\tangent{\gamma}{\fGamma}}$. Let $(h,\ell)\in\fH\times\fL$ be defined as
  \begin{align*}
    h
     & =  \dot{\rw} - \proj{\Wm}(\dot{\varphi}_1(\gamma(t)))  \in \fH,                                                                            \\
    \ell(\theta)
     & = \lambda \Big(\cJ\prt{\grad_{\fH} \cH_\theta(\varphi_2(\gamma(t))} -\dot{\varphi}_2(\gamma(t))\Big) \in \fH \quad\forall\theta\in\fTheta.
  \end{align*}
  We can express the differential as
  \begin{equation*}
    (D\cL_{\gamma})_{\dot{\gamma}}(\delta \gamma)
    = \inner*{(h, \ell), (D\varphi)_{\gamma(t)}(\delta \gamma)}_{\fH\times \fL}
    = \inner{T(h,\ell), \delta \gamma}_{\fG}, \quad \forall \delta \gamma \in \widetilde{\tangent{\gamma}{\fGamma}},
  \end{equation*}
  where we have used \Cref{lem:adjoint-like} to derive the last equality. The optimality condition is then
  $$
    \inner{T(h,\ell), \delta \gamma}_{\fG} = 0,\quad \forall \delta \gamma \in \widetilde{\tangent{\gamma}{\fGamma}}.
  $$
  Using the expression of $T(h,\ell)$ from \Cref{lem:adjoint-like} yields
  \begin{align*}
    0
     & = \inner{T(h,\ell), \delta \gamma}_{\fG}                                                                                                                         \\
     & = \inner{\inner{h, V}_{\fH},\delta a}_{\bR^{2n}}				+ \sum_{i=1}^{2n} \prt*{\int_{{\fTheta}} \delta c_i(\theta)\inner{\ell(\theta), v_i}_{\fH}\,  \mu(\d \theta)
      + \inner*{ a_i h + \int_{{\fTheta}} c_i(\theta)\ell(\theta) \mu(\d \theta) , \delta v_i }_{\fH}}
  \end{align*}
  for all $\delta \gamma = ( \delta a, \delta c, \delta V) \in \widetilde{\tangent{\gamma}{\fGamma}}$. This leads to the following system of equations for $\dot{\gamma}$: for any $1\leq i\leq 2n$,
  \begin{align}
    0= & \inner{h, v_i}_{\fH} =\inner*{\dot{\rw},v_i}_{\fH} -\sum_{j=1}^{2n}\inner*{\dot{a}_j \proj{\Wm}v_j + a_j\proj{\Wm}\dot{v}_j, v_i}_{\fH},\label{eq:opt-1}                                                                                                                         \\
    0= & \inner{\ell(\theta), v_i}_{\fH}	=\lambda \inner*{\cJ\prt{\grad_{\fH} \cH_\theta(\varphi_2(\gamma(t))},v_i}_{\fH}-\lambda\sum_{j=1}^{2n}\inner*{\dot{c}_j(\theta) v_j + c_j(\theta)\dot{v}_j, v_i }_{\fH}, \quad \forall \theta \in{\fTheta} \;\text{$\mu$-a.e.},\label{eq:opt-2} \\
    0= & \sum_{i=1}^{2n}\inner*{ a_i h + \int_{{\fTheta}} c_i(\theta)\ell(\theta) \mu(\d \theta) , \delta v_i }_{\fH}, \qquad \forall \delta V\in\widetilde{T_{V} \cV_{2n}}. \label{eq:opt-3}
  \end{align}
  Here we made all dependencies on time implicit for ease for notation.

  From \eqref{eq:opt-1} we obtain the following equation for $\dot{a} \in \bR^{2n}$:
  \begin{align*}
    \bG(V(t), \proj{\Wm} V(t)) \dot{a}(t) = \inner{\dot{\rw}(t), V}_{\fH} - \bG(V(t), \proj{\Wm}\dot{V}(t)) a(t).
  \end{align*}
  The matrix $\bG(V, \proj{\Wm} V)$ is invertible provided $\beta(\fH_{2n}, \Wm)>0$. Since this is true by assumption, we obtain \cref{eq:dynamics-dot-a}. Next, from \eqref{eq:opt-2}, using that $\inner{v_i, \dot{v}_j}_{\fH}=0$ since $\delta V \in \widetilde{\tangent{V}{\cV_{2n}}}$, we get
  \begin{align}
    \dot{c}_i(\theta) = \inner{\cJ\prt{\grad_{\fH} \cH_\theta(\varphi_2(\gamma(t))}, v_i}_{\fH}, \quad \forall \theta \in{\fTheta} \;\; \text{$\mu$-a.e.},
  \end{align}
  which is \cref{eq:dynamics-dot-c}.

  Finally, from \eqref{eq:opt-3} we obtain an equation for $\dot{v}_i$ as follows. By rearranging terms in \cref{eq:opt-3}, we can express this condition as
  \begin{align*}
    0 & = \sum_{i=1}^{2n}\inner*{a_i h + \lambda \int_{{\fTheta}} c_i(\theta) \cJ\prt{\grad_{\fH} \cH_\theta(\varphi_2(\gamma(t))} \mu(\d \theta) - \lambda \sum_{j=1}^{2n} \bC_{i,j}(c) \dot{v}_j, \delta v_i  }_{\fH}, \quad 1\leq i \leq 2n.
  \end{align*}
  By expanding the definition of $h$ and using the fact that $\delta v_i,\dot{v}_i\in\fH_{2n}^\perp$ for all $1\leq i\leq2n$, we can write the above equation as
  \begin{equation*}
    0
    =\sum_{i=1}^{2n}\inner*{\overline{\chi}_i,\delta v_i}_{\fH} \qquad \forall \delta V\in\widetilde{T_{V} \cV_{2n}}, \quad\mbox{with}\quad \overline{\chi}_i:=\chi_i-\lambda\sum_{j=1}^{2n}\bC_{i,j}(c)\dot{v}_j,
  \end{equation*}
  where $\chi_i$ is defined as in \eqref{eq:chi}.
  We then observe that, for any $\delta V\in\widetilde{T_{V} \cV_{2n}}$, it holds
  \begin{equation}\label{eq:opt-5}
    \begin{aligned}
      0=\sum_{i=1}^{2n} \inner{\overline{\chi}_i,\delta v_i}
       & =\dfrac12 \sum_{i=1}^{n}\big( \inner{\overline{\chi}_i,\delta v_i + \cJ \delta v_{i+n}}_{\fH}+\inner{\overline{\chi}_{i+n},\delta v_{i+n}-\cJ \delta v_i}_{\fH}\big) \\
       & = \dfrac12 \sum_{i=1}^{2n}\inner*{\overline{\chi}_i+\sum_{j=1}^{2n} (\bJ_{2n})_{i,j}\cJ \overline{\chi}_j,\delta v_i}_{\fH}.
    \end{aligned}
  \end{equation}
  Let us introduce the quantity
  $$\sigma_i:=\overline{\chi}_i+\sum_{j=1}^{2n} (\bJ_{2n})_{i,j}\cJ \overline{\chi}_j,\qquad \forall\, 1\leq i\leq 2n.$$
  Since $\overline{\chi}_i$ belongs to $\Vtwon^{\perp}$, it follows that $\sigma_i\in \Vtwon^{\perp}$.
  Moreover, $\sigma_i=\cJ \sigma_{i+n}$ for any $1\leq i\leq n$. Indeed,
  \begin{equation*}
    \sigma_i=\cJ\bigg(-\cJ \overline{\chi}_i+\sum_{j=1}^{2n}(\bJ_{2n})_{i,j}\overline{\chi}_j\bigg)
    =\cJ(\overline{\chi}_{i+n}-\cJ\overline{\chi}_i)
    =\cJ\bigg(\overline{\chi}_{i+n}+\sum_{j=1}^{2n}(\bJ_{2n})_{i+n,j}\cJ\overline{\chi}_j\bigg)=\cJ \sigma_{i+n}.
  \end{equation*}
  This implies that $(\sigma_1,\dots,\sigma_{2n})$ belongs to $\widetilde{T_{V} \cV_{2n}}$.
  From \eqref{eq:opt-5} we can thus conclude that $\sigma_i=0$ for any $1\leq i\leq 2n$, which gives
  \begin{equation*}
    \bS(c) \dot{V} =\lambda^{-1} (\chi+ \bJ_{2n}\cJ \chi).
  \end{equation*}
  Since $c\in \cC_{2n}$, then $\bS(c)$ is invertible. This yields \cref{eq:dynamics-dot-V} and concludes the proof.
\end{proof}

We next show that the evolution equations \eqref{eq:dynamics-dot} preserve the orthosymplecticity of the basis $V(t)$.

\begin{lemma}
  If $V(0)\in\cV_{2n}$, then $V(t)\in\cV_{2n}$ for all $t>0$.
\end{lemma}

\begin{proof}
  Since $\dot{V}\in\widetilde{T_{V} \cV_{2n}}$, then
  \begin{equation*}
    \frac{d}{dt}\inner{v_i(t),v_j(t)}_{\fH} = \inner{\dot{v}_i(t), v_j(t)}_{\fH} + \inner{v_i(t),\dot{v}_j(t)}_{\fH} = 0, \qquad 1\leq i,j\leq 2n
  \end{equation*}
  and
  \begin{equation*}
    \frac{d}{dt}\inner{\cJ v_i(t),v_j(t)}_{\fH} = \inner{\cJ\dot{v}_i(t),v_j(t)}_{\fH} + \inner{\cJ v_i(t),\dot{v}_j(t)}_{\fH} = 0, \qquad 1\leq i,j\leq2n.
  \end{equation*}
  Therefore, if $V(0)\in\cV_{2n}$,
  \begin{equation*}
    \inner{v_i(t),v_j(t)}_{\fH} = \inner{v_i(0),v_j(0)}_{\fH} = \delta_{i,j} \quad \text{ and } \quad \inner{\cJ v_i(t),v_j(t)}_{\fH}=\inner{\cJ v_i(0),v_j(0)}_{\fH}=(\bJ_{2n})_{i,j},
  \end{equation*}
  which means that $V(t)\in\cV_{2n}$ for all $t>0$.
\end{proof}

We finish this section by explaining how to appropriately choose the initial condition $\gamma(0)$ for the evolution equations \eqref{eq:adot-1}. At $t=0$, the approximation space $\fH_{2n}(0)$ is spanned by the orthosymplectic basis $V(0)$ satisfying
\begin{equation}\label{eq:V0}
  V(0)\in \argmin_{V\in\cV_{2n}}\int_{\fTheta}\norm{u(\theta)(0,\cdot)-\proj{\vspan\{V\}}(u(\theta)(0,\cdot))}^2_{\fH}\,\mu(\d\theta),
\end{equation}
where $\proj{\vspan\{V\}}$ denotes the $\fH$-orthogonal projection onto $\vspan\set{V}$. This problem can be solved using an infinite-dimensional analog of the SVD (see e.g. \cite{stewart93}, \cite[Section 6.4]{quarteroni16}).
In our numerical experiments, we select  a sample $\fTheta_h \subset \fTheta$ and construct $V(0)$ from the corresponding initial conditions of the PDE using a symplecticity-preserving SVD algorithm (see e.g. \cite{PM16}).

Given $\fH_{2n}(0)$ and $\rw(0)=\proj{W_m(0)}(u^{\dagger}(0))$,
the vector $a(0)\in\bR^{2n}$ is given by
\begin{equation}\label{eq:a0}
  a_i(0)=\inner{r(0,\cdot),v_i(0)}_{\fH}\quad\forall\,1\leq i\leq 2n,
\end{equation}
where $r(0)$ solves the least-squares problem \eqref{eq:r0} of \Cref{sec:rec-strategy}.
Finally, $c(0)\in L^2((\fTheta,\mu);\bR^{2n})$ contains the expansion coefficients
\begin{equation}\label{eq:c0}
  c_i(0)(\theta)=\inner{u(\theta)(0,\cdot),v_i(0)}_{\fH}\quad\forall\,1\leq i\leq 2n,\;\forall \theta \in \fTheta.
\end{equation}
In the numerical experiments, we ensure that $c(0)$ belongs to $\cC_{2n}$ by choosing a sufficiently large number of sample parameters in $\fTheta_h$ compared to the dimension $2n$ of the reduced basis (see \Cref{app:discrete} for details).

\subsection{The case of a symplectic observation space}

The evolution equations provided by \Cref{thm:dynamics-gamma} involve an implicit dependence on $\dot{a}$ and $\dot{V}$.
In this section we derive an equivalent system of evolution equations that are explicit in $(\dot{a},\dot{c},\dot{V})$ under the assumption that not only the reduced space, but also the observation space is a symplectic vector space. Since a finite-dimensional symplectic vector space is necessarily even-dimensional, we will denote the observation space as $\Wtwom$.

Under this additional assumption, we obtain the following result, whose proof is presented in \Cref{app:proof-gamma-dyn}.
\begin{theorem}\label{thm:dynamics-gamma-2}
  Assume that the observation space $\Wtwom\subset\fH$ is symplectic. Let $\gamma=(a,c,V)\in\widetilde{\fGamma}$. We define, for any $1\leq i\leq 2n$,
  \begin{equation}\label{eq:f}
    f_i(\gamma) := \proj{\Vtwon^\perp}\int_{\fTheta} \Big(c_i(\theta)\cJ\prt{\grad_{\fH} \cH_\theta(\varphi_2(\gamma(t))} - (\bJ_{2n}c(\theta))_i\prt{\grad_{\fH} \cH_\theta(\varphi_2(\gamma(t))}\Big)\,\mu(d\theta) \in \fH.
  \end{equation}
  If $\beta(\Vtwon(t), \Wtwom(t))>0$ for all $t>0$, then the solution $(\dot{a},\dot{c},\dot{V})\in\widetilde{\tangent{\gamma}{\fGamma}}$ to the minimization problem \eqref{eq:dynamics-gamma} is given by
  \begin{subequations}
    \label{eq:dynamics-dot-2}
    \begin{align}
      \dot{a}(t)         & = \bG(V(t), \proj{\Wtwom} V(t))^{-1} \inner{\dot{u}^\dagger-a^\top \bS^{-1}(c)f(\gamma), \proj{\Wtwom}V(t)}_{\fH} \label{eq:dynamics-dot-a-2},                                                   \\
      \dot{c}(t)(\theta) & = \inner{\cJ\prt{\grad_{\fH} \cH_\theta(\varphi_2(\gamma(t))}, V(t)}_{\fH}, \quad \forall \theta \in{\fTheta} \;\text{$\mu$-a.e.},\label{eq:dynamics-dot-c-2}                                    \\
      \dot{V}(t)         & =\bS^{-1}(c)\Big(f(\gamma) + (a+ \bJ_{2n}a\cJ) \proj{\Wtwom\cap\Vtwon^\perp}(\dot{u}^\dagger-a^\top \bS^{-1}(c)f(\gamma))(\lambda + a^\top \bS^{-1}(c) a)^{-1}\Big), \label{eq:dynamics-dot-V-2}
    \end{align}
  \end{subequations}
  where $\bS(c)$ and $\bG(V(t), \proj{\Wm} V(t))$ are defined in \eqref{eq:S-matrix} and \eqref{eq:GVPWV}, respectively.
\end{theorem}
The system is equivalent to the one derived in \Cref{thm:dynamics-gamma}, but these equations are now explicit in the time derivative of the variables.
Moreover, this formulation allows us to highlight the role of the space $\Wtwom\cap\Vtwon^\perp$ and of the penalization parameter $\lambda$: if $\lambda$ approaches infinity, then \eqref{eq:dynamics-dot-V-2} reduces to the evolution equation for the reduced basis in the forward problem \cite{P21} without any dependency on the observations, while a finite value of $\lambda>0$ introduces a correction in $\Wtwom$ that depends on the velocity measurements. We remark that, although the target velocity $\dot{u}^\dagger$ appears in \eqref{eq:dynamics-dot-a-2} and \eqref{eq:dynamics-dot-V-2}, only the knowledge of $\proj{\Wtwom}\dot{u}^\dagger$ and $\proj{\Wtwom\cap\Vtwon^\perp}\dot{u}^\dagger$ is needed to evaluate the expressions on the right-hand side, and these quantities can be computed from the measurements of $\dot{u}^\dagger$ as it can be easily deduced from \Cref{lem:PWVperp}.

\section{Error and stability analysis}\label{sec:err}
In this section, we derive a bound for the error
\begin{equation*}
  e^2_\lambda(t) \coloneqq
  \frac 1 2 \norm{u^\dagger(t) - \varphi_1(\gamma(t))}_{\fH}^2
  +
  \frac \lambda 2 \int_{{\fTheta}}\norm{u(\theta)(t,\cdot)- \varphi_2(\gamma(t))(\theta)}^2_{\fH} \, \mu(\d\theta).
\end{equation*}
The bound is presented in \Cref{thm:reconstruction-error-bound-2}. To derive it, we rely on the following intermediate result.

\begin{lemma}
  \label{lem:bound-a-c-dot-2}
  Let $\gamma(t)\in\widetilde{\fGamma}$, and let $\dot{\gamma}(t)\in\widetilde{\tangent{\gamma(t)}{\fGamma}}$ be the solution of the minimization problem \eqref{eq:dynamics-gamma}. Assume that, for any $\theta\in\fTheta$,  the operator $\cB_\theta=\cJ\prt{\grad_{\fH} \cH_\theta}$ is Lipschitz continuous with constant $L_{\theta}$. Let $L:=\sup_{\theta\in\fTheta}L_\theta$. Then, for all $t\geq 0$,
  \begin{align}
    \norm*{\dot{u}^\dagger(t)- \frac{\d}{\d t}\varphi_1(\gamma(t))}_{\fH}
     & \leq \beta^{-1}(\Vtwon, \fW_m) \min_{\dot{v} \in \Vtwon^a} \norm{\dot{u}^\dagger(t) - \dot{v} }_{\fH} \label{eq:bd1-2} \\
    \begin{split}\label{eq:bd2-2}
      \norm*{ \dot{u}(\theta)(t,\cdot)- \frac{\d}{\d t}\varphi_2(\gamma(t))(\theta)}^2_{\fL} &\leq L^2\norm*{u(\theta)(t,\cdot)-\varphi_2(\gamma(t))(\theta)}_{\fL}^2 \\&+ \int_{\fTheta}\min_{h(\theta)\in\Vtwon^c(\theta)}\norm*{\dot{u}(\theta)(t,\cdot)-h(\theta)}_{\fH}^2\,\mu(d\theta)\end{split}
  \end{align}
  where
  \begin{align*}
    \Vtwon^a         & \coloneqq \prt*{\sum_{i=1}^{2n} a_i \dot{v}_i}+ \Vtwon                                             \\
    \Vtwon^c(\theta) & \coloneqq \prt*{\sum_{i=1}^{2n} c_i(\theta) \dot{v}_i}+ \Vtwon, \quad \forall \theta\in {\fTheta}.
  \end{align*}
\end{lemma}
\begin{proof}
  Let $t>0$. The minimization problem \eqref{eq:dynamics-gamma} can be written as
  \begin{equation*}
    \min_{\delta V \in \widetilde{T_{V(t)} \cV_{2n}}}\left( \min_{\delta a \in \bR^{2n}} \cL_{\gamma,1}(\delta a, \delta V) + \lambda \int_{{\fTheta}} \min_{\delta c(\theta)\in\bR^{2n}} \cL_{\gamma,2}(\delta c(\theta), \delta V)  \mu(\d \theta)\right)
  \end{equation*}
  with
  \begin{align*}
    \cL_{\gamma,1}(\delta a, \delta V)         & \coloneqq \frac 1 2 \norm*{ \dot{\rw}(t) - \proj{\Wm}\prt*{\sum_{i=1}^{2n}\prt{\delta a_i v_i(t) + a_i(t)\delta v_i}}}^2_{\fH},                            \\
    \cL_{\gamma,2}(\delta c(\theta), \delta V) & \coloneqq \frac 1 2 \norm*{\cB_\theta(\varphi_2(\gamma(t))(\theta)) - \sum_{i=1}^{2n}\prt*{\delta c_i(\theta) v_i(t) + c_i(t)(\theta)\delta v_i}}^2_{\fH}.
  \end{align*}
  For a fixed $\delta V$,
  the minimizer $\dot{a}=\argmin_{\delta a \in \bR^{2n}} \cL_{\gamma,1}(\delta a, \delta V)$
  satisfies (see, e.g., \cite[Theorem 2.9]{BCDDPW17})
  \begin{equation*}
    \norm*{\dot{u}^\dagger(t)-\sum_{i=1}^{2n}(a_i(t)\delta v_i+\dot{a}_iv_i(t))}_{\fH}\leq\beta^{-1}(\Vtwon,\Wm)\norm*{\proj{\Vtwon^\perp}\left(\dot{u}^\dagger(t)-\sum_{i=1}^{2n}a_i(t)\delta v_i\right)}_{\fH}.
  \end{equation*}
  Then, setting $\delta v_i=\dot{v}_i$, for all $1\leq i\leq 2n$, yields
  \eqref{eq:bd1-2}; thereby
  \begin{align*}
    \norm*{\dot{u}^\dagger(t)- \frac{\d}{\d t}\varphi_1(\gamma(t))}_{\fH} & =\norm*{\dot{u}^\dagger(t)-\sum_{i=1}^{2n}(a_i(t)\dot{v}_i+\dot{a}_iv_i(t))}_{\fH} \\&\leq\beta^{-1}(\Vtwon,\Wm)\norm*{\proj{\Vtwon^\perp}\left(\dot{u}^\dagger(t)-\sum_{i=1}^{2n}a_i(t) \dot{v}_i\right)}_{\fH}\\&=\beta^{-1}(\Vtwon,\Wm)\inf_{\dot{v}\in\Vtwon^a}\norm{\dot{u}^\dagger(t)-\dot{v}}_{\fH}.
  \end{align*}
  Similarly, for fixed $\delta V$ and $\theta\in\fTheta$,
  the minimizer $\dot{c}(\theta)=\argmin_{\delta c(\theta)\in\bR^{2n}} \cL_{\gamma,2}(\delta c(\theta), \delta V)$
  satisfies $\dot{c}_i(\theta)=\inner{\cB_\theta(\varphi_2(\gamma(t))(\theta)),v_i}_{\fH}$ $\mu$-a.e., for any $1\leq i\leq 2n$.
  Therefore,
  \begin{align*}
    \norm*{ \dot{u}(\theta)(t,\cdot)- \frac{\d}{\d t}\varphi_2(\gamma(t))(\theta) }^2_{\fL} & = \int_{\fTheta}\norm*{\dot{u}(\theta)(t,\cdot) - \sum_{i=1}^{2n}\left(\dot{c}_i(\theta)v_i(t)+c_i(t)(\theta)\dot{v}_i\right)}^2_{\fH}\,\mu(d\theta) \\&=\int_{\fTheta}\norm*{\dot{u}(\theta)(t,\cdot) - \proj{\Vtwon}\cB_\theta(\varphi_2(\gamma(t))(\theta)) - \sum_{i=1}^{2n}c_i(t)(\theta)\dot{v}_i}^2_{\fH}\,\mu(d\theta).
  \end{align*}
  We then decompose the term $\dot{u}(\theta)(t,\cdot)$ into its projections onto $\Vtwon$ and $\Vtwon^\perp$, and we write the integrand as
  \begin{align*}
     & \norm*{\proj{\Vtwon}\Big(\cB_\theta(u(\theta)(t,\cdot))-\cB_\theta(\varphi_2(\gamma(t))(\theta))\Big)}_{\fH}^2 + \norm*{\proj{\Vtwon^\perp}\dot{u}(\theta)(t,\cdot)-\sum_{i=1}^{2n}c_i(t)(\theta)\dot{v}_i}_{\fH}^2 \\&\leq L^2\norm*{u(\theta)(t,\cdot)-\varphi_2(\gamma(t)(\theta))}^2_{\fH} + \norm*{\proj{\Vtwon^\perp}\dot{u}(\theta)(t,\cdot)-\sum_{i=1}^{2n}c_i(t)(\theta)\dot{v}_i}_{\fH}^2\\&=L^2\norm*{u(\theta)(t,\cdot)-\varphi_2(\gamma(t))(\theta)}^2_{\fH}+\inf_{h(\theta)\in\Vtwon^c(\theta)}\norm*{\dot{u}(\theta)(t,\cdot)-h(\theta)}^2_{\fH}.
  \end{align*}
  Integrating over $\fTheta$ yields \eqref{eq:bd2-2}.
\end{proof}

\begin{theorem}
  \label{thm:reconstruction-error-bound-2}
  For every $t\geq0$,
  \begin{equation}
    \label{eq:reconstruction-error-bound-2}
    e_\lambda(t)\leq e_\lambda(0)\exp\left(\sqrt{2}Lt\right) + \int_0^t C(s)\exp\left(\sqrt{2}L(t-s)\right)\d s
  \end{equation}
  with
  \begin{equation*}
    C(t) \coloneqq \prt*{\beta^{-2}(\Vtwon(t), \fW_m(t)) \min_{\dot{v} \in \Vtwon^a(t)} \norm{\dot{u}^\dagger(t) - \dot{v} }^2_{\fH} + \lambda \int_{{\fTheta}} \min_{h(\theta)\in \Vtwon^c(t,\theta)}\norm{ \dot{u}(\theta)(t,\cdot) -h(\theta)}^2_{\fH} \mu(\d \theta)  }^{1/2}
  \end{equation*}
  and $\Vtwon^a$ and $\Vtwon^c(\theta)$ are the affine spaces introduced in \Cref{lem:bound-a-c-dot-2}.
\end{theorem}
\begin{proof}
  Let $t\geq 0$. We have
  \begin{equation}\label{eq:err-bound-intermediate-1-2}
    \begin{aligned}
      \frac{\d e^2_\lambda}{\d t}(t)
       & =
      \inner*{\dot{u}^\dagger(t)- \frac{\d}{\d t}\varphi_1(\gamma(t)), u^\dagger(t) - \varphi_1(\gamma(t))}_{\fH}                                                                                                            \\
       & \quad + \lambda
      \int_{{\fTheta}} \inner*{ \dot{u}(\theta)(t,\cdot)- \frac{\d}{\d t}\varphi_2(\gamma(t))(\theta), u(\theta)(t,\cdot)- \varphi_2(\gamma(t))(\theta)}_{\fH} \mu(\d \theta)                                                \\
       & \leq \norm*{\dot{u}^\dagger(t)- \frac{\d}{\d t}\varphi_1(\gamma(t))}_{\fH}
      \norm*{u^\dagger(t) - \varphi_1(\gamma(t))}_{\fH}                                                                                                                                                                      \\
       & \quad + \lambda
      \norm*{ \dot{u}(\theta)(t,\cdot)- \frac{\d}{\d t}\varphi_2(\gamma(t))(\theta) }_{\fL}
      \norm*{u(\theta)(t,\cdot)- \varphi_2(\gamma(t))(\theta)}_{\fL}                                                                                                                                                         \\
       & \leq \prt*{ \norm*{\dot{u}^\dagger(t)- \frac{\d}{\d t}\varphi_1(\gamma(t))}^2_{\fH} + \lambda \norm*{ \dot{u}(\theta)(t,\cdot)- \frac{\d}{\d t}\varphi_2(\gamma(t))(\theta) }_{\fL}^2 }^{1/2} \sqrt{2}e_\lambda(t).
    \end{aligned}
  \end{equation}
  We next apply \Cref{lem:bound-a-c-dot-2} to bound the two terms inside the square root of \eqref{eq:err-bound-intermediate-1-2} and the inequality $\sqrt{\alpha+\gamma}\leq\sqrt{2}(\sqrt{\alpha}+\sqrt{\gamma})$ for all $\alpha,\gamma\geq0$, which directly leads to
  \begin{equation*}
    \frac{\d e^2_\lambda}{\d t}(t)\leq 2C(t)e_\lambda(t) + 2\sqrt{2}Le^2_\lambda(t).
  \end{equation*}
  By a Grönwall estimate (see, e.g., \cite[Theorem 2]{WW65}), we obtain the error bound \eqref{eq:reconstruction-error-bound-2}.
\end{proof}

\section{Dynamical sensor placement}
\label{sec:moving-sensors}
\Cref{thm:reconstruction-error-bound-2} shows that the quality of the reconstruction deteriorates if the stability constant $\beta(\Vtwon(t),\Wtwom)$ approaches zero. On the other hand, when the reduced space $\Vtwon$ evolves in time, it is possible to construct simple examples where $\beta(\Vtwon(t),\Wtwom)\to0$ as $t\to\infty$, see e.g. \cite[Section 2.2]{MPV25}. This motivates the development of techniques to also evolve the observation space $\Wtwom$ in time. In this work we follow the same procedure as in \cite{MPV25}, that we briefly recall here for completeness. The gist of the method is to connect the tasks of evolving the observation space $\Wtwom(t)$ and dynamically updating the positions of the sensors, and to define an evolution of the sensors' locations that maximizes the stability constant $\beta(\Vtwon(t), \Wtwom(t))$ at all times.

In the following, we denote the location of the $j$-th sensors at time $t$ as $\overline{x}_j(t)\in\fD$, for $1\leq j\leq m$, and we collect the sensors' locations in the vector $\overline{x}(t) = (\overline{x}_1(t),\dots,\overline{x}_m(t))\in\fD^m$. We then assume that the Riesz representer $\rw_i$ of the $i$th measurement $\ell_i$ is a function that depends on the location of the $i$th sensor, so that we can write $\rw_i(t) = \rw_i(\overline{x}_i(t))$. As a guiding example, we consider in this work the case of local averages of width $\sigma>0$ around $\overline{x}_i(t)$: if $\fH=L^2(\fD)\times L^2(\fD)$ with $\fD\subset\bR^d$, we set
\begin{equation*}
  \widetilde{\rw}_i(t)(x) := \frac{1}{(2\pi\sigma^2)^{d/2}}\exp\left(-\frac{\norm{x-\overline{x}_i(t)}_{2}}{2\sigma^2}\right), \qquad 1\leq i\leq m
\end{equation*}
and we define $\Wtwom(t)=\vspan\{W(t)\}$, where $W(t)=(\rw_1(t),\dots,\rw_{2m}(t))\in\fH^{2m}$ and
\begin{equation*}
  \rw_i(t) = \begin{cases}
    (\widetilde{\rw}_i(t), 0)     & 1\leq i\leq m    \\
    (0, \widetilde{\rw}_{i-m}(t)) & m+1\leq i \leq2m
  \end{cases}.
\end{equation*}
This means that every sensor yields a measurement of both components of the target function in $\fH$. Moreover, observe that $\Wtwom(t)$ is a symplectic vector space at all times.
Since $\widetilde{\rw}_i$ depends on time only through the position of the $i$th sensor, we shall write $\widetilde{\rw}_i(t)=\widetilde{\rw}_i(\overline{x}_i(t))$, and, consequently, $W(t)=W(\overline{x}(t))$ and $\Wtwom(t) = \Wtwom(\overline{x}(t))$. Assuming that the evolution of the reduced space $\Vtwon(t)$ is prescribed, we can then regard the stability constant $\beta$ as a function of the sensors' locations: $\beta(\overline{x}(t))=\beta(\Vtwon(t), \Wtwom(\overline{x}(t)))$. Then, in view of \Cref{thm:reconstruction-error-bound-2}, we define the trajectory $\overline{x}(t)$ as the one that maximizes $\beta$.

In practice, this procedure boils down to the following time-discrete scheme. Assume that the location of the sensors at time $t_j$ is given, and that the reduced space is updated from $\Vtwon(t_j)$ to $\Vtwon(t_{j+1})$. We want to determine the new locations $\overline{x}(t_{j+1})\in\fD^{m}$ by maximizing $\beta(\Vtwon(t_{j+1}), \Wtwom(\cdot))$. One possibility to solve this problem is by means of a gradient descent method with the old locations $\overline{x}(t_j)$ as an initial guess, as outlined in \Cref{alg:sensors_update}. We refer to \cite{MPV25} for a detailed discussion of this procedure, as well as for the expression of the gradient $\nabla_{\overline{x}}\beta$ appearing at line 6.
\begin{algorithm}[H]
  \caption{Update the position of the sensors}\label{alg:sensors_update}
  \begin{algorithmic}[1]
    \Procedure{$\overline{x}(t_{j+1})$ = \textsc{Sensors update}}{$\overline{x}(t_j)$, $\Vtwon(t_{j+1})$}
    \State $\overline{x}^{(0)}\gets\overline{x}(t_j)$
    \State $l\gets0$
    \While{$l<l_{\text{max}}$}
    \State Determine step size $\alpha_l$
    \State $\overline{x}^{(l+1)} = \overline{x}^{(l)} + \alpha_l\nabla_{\overline{x}}\beta(\Vtwon(t_{j+1}),\Wtwom(\overline{x}^{(l)}))$
    \State $l\gets l+1$
    \EndWhile
    \State $\overline{x}(t_{j+1}) = \overline{x}^{(l)}$
    \EndProcedure
  \end{algorithmic}
\end{algorithm}

\section{Numerical experiments}\label{sec:num-exp}

In this section we assess the performance of the reconstruction algorithm on the two-dimensional shallow water equations
\begin{equation}\label{eq:SWE}
  \begin{cases}
    \partial_t h + \nu\nabla\cdot(h\nabla\Phi) = 0                     & (t,x)\in (0,T]\times \fD, \\
    \partial_t\Phi + \frac{\nu}{2}\lvert\nabla\Phi\rvert^2 + \nu h = 0 & (t,x)\in (0,T]\times \fD,
  \end{cases}
\end{equation}
with periodic boundary conditions, where $x=(x_1,x_2)$ and the spatial domain $\fD\subset\Rbb^2$ is the rectangle $[-L_{x_1},L_{x_1}]\times[-L_{x_2},L_{x_2}]$. The initial condition is  $h(\theta)(0,x)=1+\frac{1}{2}e^{-\alpha\lVert x\rVert^2}$ and  $\Phi(\theta)(0,x)=0$. The parameter $\theta=(\alpha,\nu)$ belongs to the parameter space $\fTheta=[1.1,1.7]\times[0.8,1]\subset\Rbb^2$.

The shallow water equations admit a Hamiltonian formulation in $\fH = L^2(\fD)\times L^2(\fD)$ where the Hamiltonian is defined as
\begin{equation*}
  \Hcal_\nu(u) =
  \begin{cases}
    \frac{\nu}{2}\int_{\fD} h(x)\left(\lvert\nabla\Phi(x))\rvert^2+h(x)\right)\,dx, & \quad \forall u=(h, \Phi)\in \dom(\cH)\coloneqq L^2(\fD)\times H^1(\fD) \\
    +\infty,                                                                        & \quad \forall u \in \fH\setminus\dom(\cH).
  \end{cases}
\end{equation*}
We assume that, at each time $t$ and for each parameter $\theta$, the functions $h(\theta)(t,\cdot)\in L^2(\fD)$ and $\Phi(\theta)(t,\cdot)\in L^2(\fD)$ solutions of \eqref{eq:SWE} are sufficiently regular so that they admit classical derivatives, and $\grad_{\fH}\Hcal$ is well-defined. We then employ a finite difference scheme for spatial discretization on a uniform computational grid of $N_{x_1}$ and $N_{x_2}$ intervals in each spatial direction.

In the numerical experiments we choose $L_{x_1}=L_{x_2}=8$ and $N_{x_1}=N_{x_2}=50$, which corresponds to $\Delta {x_1}=\Delta x_2 = 0.32$ and $N=N_{x_1}N_{x_2}=5000$. The final time is $T=7$ and we employ $N_t=3500$ time steps for temporal discretization, so that $\Delta t=2\times10^{-3}$. The parameter space $\fTheta$ is discretized using $10$ uniformly sampled parameters in each direction, which gives $p=100$. The true parameter is $\theta^\dagger=(1.3616,0.8871)$, and we consider local averages with a finite width $\sigma=0.1$ for the measurement of the velocity $\dot{u}^\dagger=\Dot{u}(\theta^\dagger)$ as discussed in \Cref{sec:moving-sensors}.
We set $n=6$, $m=8$ and $\lambda=0.1$. Qualitatively similar results were observed for different values of $\lambda$.

We consider three different test cases. In the first scenario, the sensors are randomly placed around the initial profile of the solution, and are not moved during the simulation. In a second test, we consider the same setup with static measurements, but we arrange the sensors on a uniform grid instead. Finally, in the third test, the sensors are moved at each time step as discussed in \Cref{sec:moving-sensors}. We compare the three cases by analyzing the quality of the reconstruction of the solution associated to the true parameter $\theta^\dagger$ and the accuracy of the approximation across the whole parameter space $\fTheta$. To this end we introduce the relative errors
\begin{equation}\label{eq:rec_error}
  e_{\varphi_1}(t) := \frac{\lVert {u}^\dagger(t) - \varphi_1(\gamma(t))\rVert_{\fH}}{\lVert u^\dagger(t)\rVert_{\fH}} \qquad \text{and} \qquad e_{\varphi_2}(t) := \frac{\lVert {u}(\theta)(t,\cdot) - \varphi_2(\gamma(t))(\theta)\rVert_{\fL}}{\lVert u(\theta)(t,\cdot)\rVert_{\fL}},
\end{equation}
respectively. We also consider the error
\begin{equation}\label{eq:Ham_error}
  e_\Hcal(t) := \frac{\lvert\Hcal_{\theta^\dagger}(\varphi_1(\gamma(t)))-\Hcal_{\theta^\dagger}(\varphi_1(\gamma(0)))\rvert}{\lvert\Hcal_{\theta^\dagger}(\varphi_1(\gamma(0)))\rvert},
\end{equation}
that measures the variation of the Hamiltonian associated with the exact parameter $\theta^\dagger$ compared to its initial value when evaluated at the reconstructed solution.
Moreover, in order to assess the quality of the reconstruction, we compute the error
\begin{equation}\label{eq:bestappr_error}
  \underline{e_{\varphi_1}}(t) := \frac{\lVert u^\dagger(t) - \proj{\Vtwon(t)}u^\dagger(t)\rVert_{\fH}}{\lVert u^\dagger(t)\rVert_{\fH}},
\end{equation}
which is a measure of the best possible approximation of the exact solution $u^\dagger(t)$ by an element of $\Vtwon(t)$ at each time. We remark that $\underline{e_{\varphi_1}}(t)\leq e_{\varphi_1}(t)$ for all $t\in[0,T]$. Finally, we also monitor the time evolution of the stability constant $\beta(\Vtwon(t),\Wtwom(t))$ appearing in \Cref{thm:reconstruction-error-bound-2}.

\paragraph{Static sensors, random placement.}
We consider the case of static measurements first: the sensors' locations are chosen randomly in the square $[-1,1]\times[-1,1]$, around the initial solution profile, and kept fixed throughout the simulation. We show in \Cref{fig:static_solution} a comparison between the exact solution (first column) and the reconstruction (second column) at the final time. For ease of comparison, we also report the absolute value of the difference in the third column. We observe that the reconstructed solution is qualitatively different from the target, and in particular it completely fails to capture the amplitude of the height $h$.
\begin{figure}[H]
  \centering
  \makebox[\textwidth]{
    \includegraphics[width=1.2\linewidth]{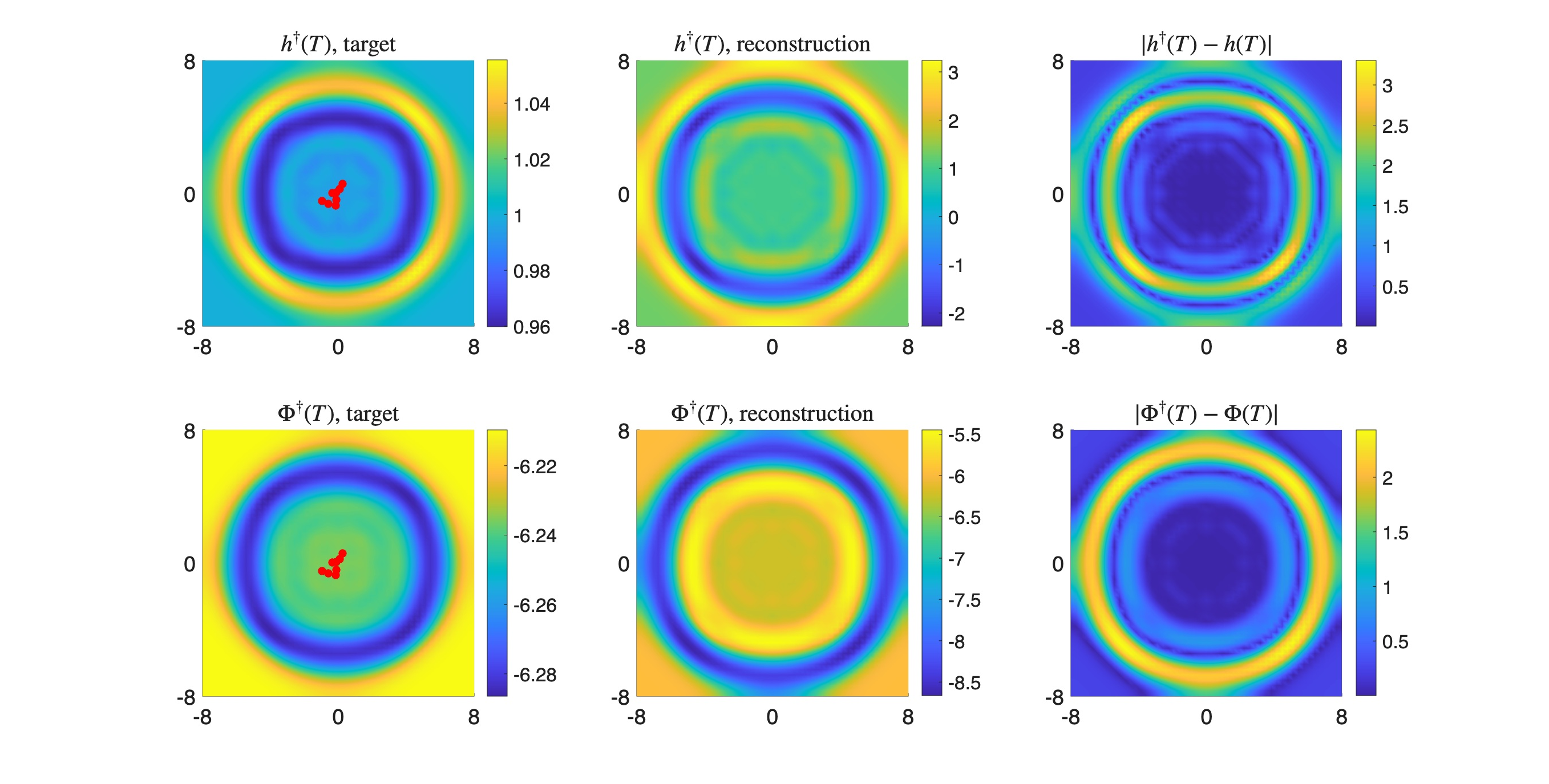}
  }
  \caption{\footnotesize Reconstruction with static, randomly placed sensors. Comparison between the height $h^\dagger(T)$ and the velocity potential $\Phi^\dagger(T)$ associated with the exact parameter at the final time (left column) and the reconstructed quantities (center column), and absolute value of the difference between target and reconstruction (right column). The red dots in the left column represent the locations of the sensors.}
  \label{fig:static_solution}
\end{figure}
The evolution of the reconstruction error over time is reported in \Cref{fig:static_error_beta}, together with the stability constant $\beta(\Vtwon(t),\Wtwom)$. The reconstruction error is significantly larger than the best approximation error at all times, and the value of $\beta(\Vtwon(t),\Wtwom)$ is around $10^{-5}$ at the final time. This has also a negative effect on the preservation of the Hamiltonian evaluated at the reconstructed solution, as shown in \Cref{fig:static_Hamiltonian}.

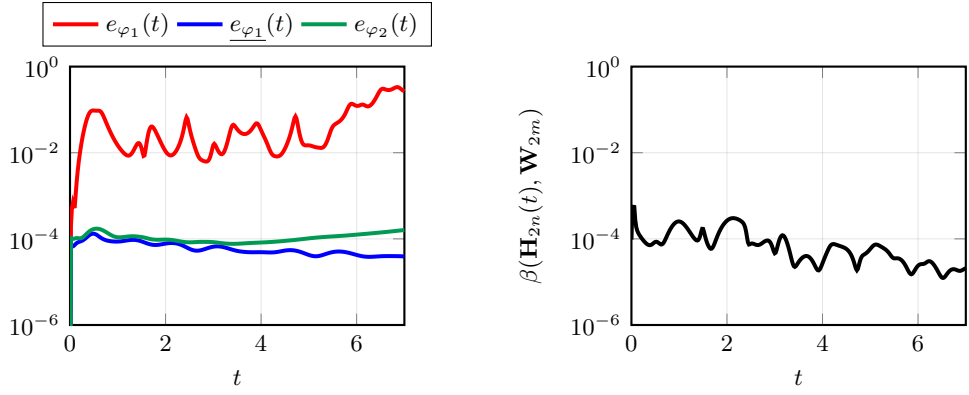
\begin{figure}[H]
  \centering
  \begin{tikzpicture}
    \begin{groupplot}[
        group style={group size=2 by 1,
            horizontal sep=3cm},
        width=6cm, height=5cm,
      ]
      \nextgroupplot[
        xlabel={$t$},
        ylabel style = {yshift=-.2cm},
        axis line style = thick,
        grid=both,
        minor tick num=0,
        minor tick style={draw=none},
        ytick={1e-08,1e-06,1e-04,1e-02,1e-00},
        ymode=log,
        grid style = {gray,opacity=0.2},
        xmin=0, xmax=7,
        ymin=1e-06, ymax=1e-00,
        xlabel style={font=\footnotesize},
        ylabel style={font=\footnotesize},
        x tick label style={font=\footnotesize},
        y tick label style={font=\footnotesize},
        legend style={font=\footnotesize},
        legend cell align={left},
        legend columns = 3,
        legend style={at={(0.5,1.25)},anchor=north}]
      \addplot+[color=red,mark=none,line width=1.5pt] table[x=t,y=e_phi1] {data/static_error_beta.txt};
      \addplot+[color=blue,mark=none,line width=1.5pt] table[x=t,y=e_phi1_und] {data/static_error_beta.txt};
      \addplot+[color=ForestGreen,mark=none,line width=1.5pt] table[x=t,y=e_phi2] {data/static_error_beta.txt};
      \legend{{$e_{\varphi_1}(t)$},{$\underline{e_{\varphi_1}}(t)$},{$e_{\varphi_2}(t)$}};
      \nextgroupplot[ylabel={$\beta(\Vtwon(t),\Wtwom)$},
        xlabel={$t$},
        ylabel style = {yshift=.1cm},
        axis line style = thick,
        grid=both,
        minor tick num=0,
        ymode = log,
        grid style = {gray,opacity=0.2},
        xmin=0, xmax=7,
        ymin=1e-06, ymax=1,
        ytick={1e-06,1e-04,1e-02,1e-00},
        xlabel style={font=\footnotesize},
        ylabel style={font=\footnotesize},
        x tick label style={font=\footnotesize},
        y tick label style={font=\footnotesize}]
      \addplot+[color=black,mark=none,line width=1.5pt] table[x=t,y=beta] {data/static_error_beta.txt};
    \end{groupplot}
  \end{tikzpicture}
  \caption{\footnotesize Reconstruction with static, randomly placed sensors. Left: evolution of the relative reconstruction error $e_{\varphi_1}(t)$ and the relative approximation error $e_{\varphi_2}(t)$ from \eqref{eq:rec_error}, and the best approximation error $\underline{e_{\varphi_1}}(t)$ from \eqref{eq:bestappr_error}. Right: evolution of the stability constant $\beta(\Vtwon(t),\Wtwom)$.}\label{fig:static_error_beta}
\end{figure}

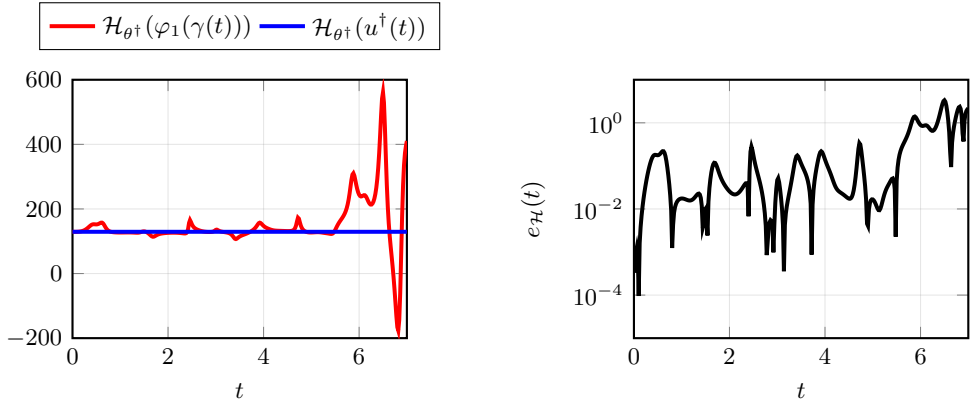
\begin{figure}[H]
  \centering
  \begin{tikzpicture}
    \begin{groupplot}[
        group style={group size=2 by 1,
            horizontal sep=3cm},
        width=6cm, height=5cm,
      ]
      \nextgroupplot[ylabel={},
        xlabel={$t$},
        ylabel style = {yshift=-.2cm},
        axis line style = thick,
        grid=both,
        minor tick num=0,
        minor tick style={draw=none},
        grid style = {gray,opacity=0.2},
        xmin=0, xmax=7,
        ymin=-200, ymax=600,
        xlabel style={font=\footnotesize},
        ylabel style={font=\footnotesize},
        x tick label style={font=\footnotesize},
        y tick label style={font=\footnotesize},
        legend style={font=\footnotesize},
        legend cell align={left},
        legend columns = 2,
        legend style={at={(0.5,1.3)},anchor=north}]
      \addplot+[color=red,mark=none,line width=1.5pt] table[x=t,y=Hustar] {data/static_Hamiltonian.txt};
      \addplot+[color=blue,mark=none,line width=1.5pt] table[x=t,y=Hu] {data/static_Hamiltonian.txt};
      \legend{{$\Hcal_{\theta^\dagger}(\varphi_1(\gamma(t)))$},{$\Hcal_{\theta^\dagger}(u^\dagger(t))$}};
      \nextgroupplot[ylabel={$e_\Hcal(t)$},
        xlabel={$t$},
        ylabel style = {yshift=.1cm},
        axis line style = thick,
        grid=both,
        minor tick num=0,
        ymode = log,
        grid style = {gray,opacity=0.2},
        xmin=0, xmax=7,
        ymin=1e-05, ymax=10,
        ytick={1e-08,1e-06,1e-04,1e-02,1e-00},
        xlabel style={font=\footnotesize},
        ylabel style={font=\footnotesize},
        x tick label style={font=\footnotesize},
        y tick label style={font=\footnotesize}]
      \addplot+[color=black,mark=none,line width=1.5pt] table[x=t,y=EH] {data/static_Hamiltonian.txt};
    \end{groupplot}
  \end{tikzpicture}
  \caption{\footnotesize Reconstruction with static, randomly placed sensors. Left: Hamiltonian evaluated at the exact solution $u^\dagger(t)$ and at the reconstructed solution $\varphi_1(\gamma(t))$. Right: evolution of the error in the Hamiltonian conservation $e_\Hcal(t)$ \eqref{eq:Ham_error}.}\label{fig:static_Hamiltonian}
\end{figure}

\paragraph{Static sensors, uniform placement.}
We now consider the same setup of $m=8$ static measurements, which are now placed on a uniform grid around the initial profile of $h$. In a general procedure, when the number of sensors satisfies $m=m_{x_1}m_{x_2}$ for $m_{x_1}, m_{x_2}\in\Nbb$, one can consider a rectangle $[-\widehat{L}_{x_1}, \widehat{L}_{x_1}]\times [-\widehat{L}_{x_2}, \widehat{L}_{x_2}]$, with $\widehat{L}_{x_1}\leq L_{x_1}$, $\widehat{L}_{x_2}\leq L_{x_2}$ and place the sensors on the inner vertices of the grid resulting from the discretization by uniform step sizes $\widehat{d}_{x_i}:=2\widehat{L}_{x_i}/(m_{x_i}+1)$, for $i\in\{1,2\}$. Numerical experiments actually reveal that a symmetric arrangement of sensors around the center of the computational domain $\fD$ results in a rank-deficient matrix $\Gbb$, i.e., the stability constant $\beta$ vanishes. This is due to the symmetry of the target solution around the origin. One possible remedy to this issue is to break the symmetry and ensure that the center of the sensors' arrangement does not coincide with the center of $\fD$. In other words, we consider the rectangle $[-\widehat{L}_{x_1}-\varepsilon_{x_1}, \widehat{L}_{x_1}-\varepsilon_{x_1}]\times [-\widehat{L}_{x_2}-\varepsilon_{x_2}, \widehat{L}_{x_2}-\varepsilon_{x_2}]$ for some $\varepsilon_{x_i}\neq0$, $i\in\{1,2\}$. In the particular experiment we discuss in this section, we choose $m_{x_1}=4$, $m_{x_2}=2$, $\widehat{L}_{x_i} = L_{x_i}/3$ and $\varepsilon_{x_i}=\widehat{d}_{x_i}/10$. Although only this configuration is presented here, we observed qualitatively similar results with different choices of $m_{x_i}$, $\widehat{L}_{x_i}$ and $\varepsilon_{x_i}$.

We first plot the reconstruction at the final time and a comparison with the target in \Cref{fig:static_unif_solution}. Although the maximum value of the pointwise error is smaller than in the case of randomly placed sensors, being around $10\%$ as shown in the third column, the reconstructed profile of both the height $h$ and the potential $\Phi$ is completely different from the target.

\begin{figure}[H]
  \centering
  \makebox[\textwidth]{
    \includegraphics[width=1.2\linewidth]{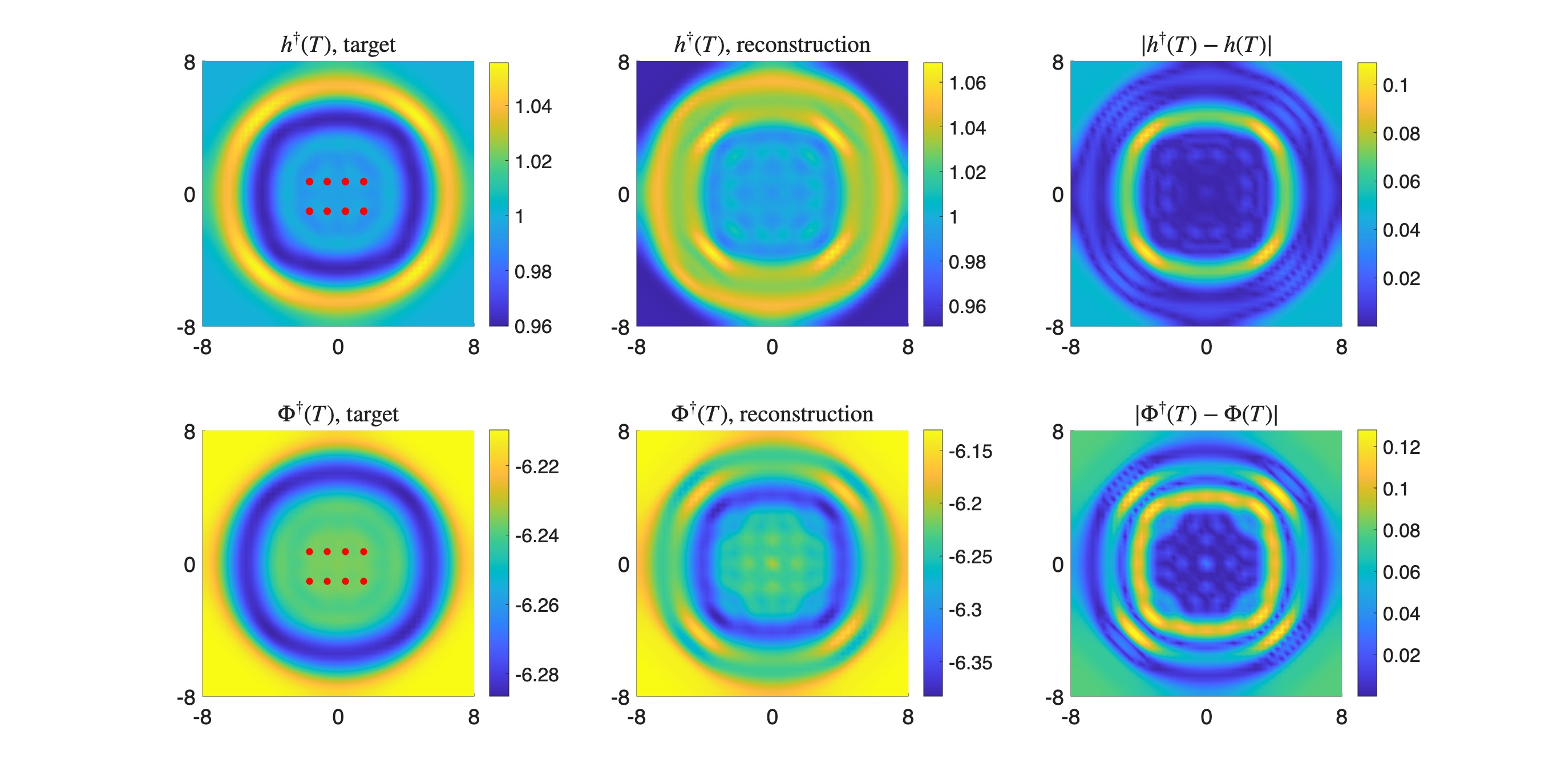}
  }
  \caption{\footnotesize Reconstruction with static, uniformly placed sensors. Comparison between the height $h^\dagger(T)$ and the velocity potential $\Phi^\dagger(T)$ associated with the exact parameter at the final time (left column) and the reconstructed quantities (center column), and absolute value of the difference between target and reconstruction (right column). The red dots in the left column represent the locations of the sensors.}
  \label{fig:static_unif_solution}
\end{figure}

The relative reconstruction error in the $\fH$-norm remains above $1\%$ in the final stages of the simulation, as reported in \Cref{fig:static_unif_error_beta}.  We also observe that the value of the stability constant $\beta(\Vtwon(t),\Wtwom)$ is smaller than $10^{-3}$ for most of the time interval $[0,T]$. Finally, \Cref{fig:static_unif_Hamiltonian} shows that the relative error in the Hamiltonian conservation can be as large as $1$. The results of this test confirm that a uniform arrangement of sensors, while reasonable and easily achievable in practice, does not correspond to the optimal value of the stability constant. Moreover, it can result in rank-deficient matrices in the case of symmetric solutions.

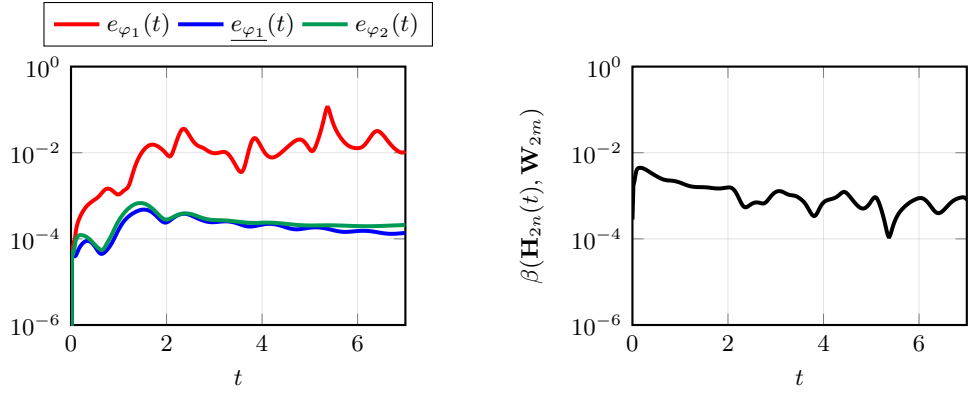
\begin{figure}[H]
  \centering
  \begin{tikzpicture}
    \begin{groupplot}[
        group style={group size=2 by 1,
            horizontal sep=3cm},
        width=6cm, height=5cm,
      ]
      \nextgroupplot[
        xlabel={$t$},
        ylabel style = {yshift=-.2cm},
        axis line style = thick,
        grid=both,
        minor tick num=0,
        minor tick style={draw=none},
        ytick={1e-08,1e-06,1e-04,1e-02,1e-00},
        ymode=log,
        grid style = {gray,opacity=0.2},
        xmin=0, xmax=7,
        ymin=1e-06, ymax=1e-00,
        xlabel style={font=\footnotesize},
        ylabel style={font=\footnotesize},
        x tick label style={font=\footnotesize},
        y tick label style={font=\footnotesize},
        legend style={font=\footnotesize},
        legend cell align={left},
        legend columns = 3,
        legend style={at={(0.5,1.25)},anchor=north}]
      \addplot+[color=red,mark=none,line width=1.5pt] table[x=t,y=e_phi1] {data/static_unif_error_beta.txt};
      \addplot+[color=blue,mark=none,line width=1.5pt] table[x=t,y=e_phi1_und] {data/static_unif_error_beta.txt};
      \addplot+[color=ForestGreen,mark=none,line width=1.5pt] table[x=t,y=e_phi2] {data/static_unif_error_beta.txt};
      \legend{{$e_{\varphi_1}(t)$},{$\underline{e_{\varphi_1}}(t)$},{$e_{\varphi_2}(t)$}};
      \nextgroupplot[ylabel={$\beta(\Vtwon(t),\Wtwom)$},
        xlabel={$t$},
        ylabel style = {yshift=.1cm},
        axis line style = thick,
        grid=both,
        minor tick num=0,
        ymode = log,
        grid style = {gray,opacity=0.2},
        xmin=0, xmax=7,
        ymin=1e-06, ymax=1,
        ytick={1e-06,1e-04,1e-02,1e-00},
        xlabel style={font=\footnotesize},
        ylabel style={font=\footnotesize},
        x tick label style={font=\footnotesize},
        y tick label style={font=\footnotesize}]
      \addplot+[color=black,mark=none,line width=1.5pt] table[x=t,y=beta] {data/static_unif_error_beta.txt};
    \end{groupplot}
  \end{tikzpicture}
  \caption{\footnotesize Reconstruction with static, uniformly placed sensors. Left: evolution of the relative reconstruction error $e_{\varphi_1}(t)$ and the relative approximation error $e_{\varphi_2}(t)$ from \eqref{eq:rec_error}, and the best approximation error $\underline{e_{\varphi_1}}(t)$ from \eqref{eq:bestappr_error}. Right: evolution of the stability constant $\beta(\Vtwon(t),\Wtwom)$.}\label{fig:static_unif_error_beta}
\end{figure}

\begin{figure}[H]
  \centering
  \begin{tikzpicture}
    \begin{groupplot}[
        group style={group size=2 by 1,
            horizontal sep=3cm},
        width=6cm, height=5cm,
      ]
      \nextgroupplot[ylabel={},
        xlabel={$t$},
        ylabel style = {yshift=-.2cm},
        axis line style = thick,
        grid=both,
        minor tick num=0,
        minor tick style={draw=none},
        grid style = {gray,opacity=0.2},
        xmin=0, xmax=7,
        ymin=-200, ymax=600,
        xlabel style={font=\footnotesize},
        ylabel style={font=\footnotesize},
        x tick label style={font=\footnotesize},
        y tick label style={font=\footnotesize},
        legend style={font=\footnotesize},
        legend cell align={left},
        legend columns = 2,
        legend style={at={(0.5,1.3)},anchor=north}]
      \addplot+[color=red,mark=none,line width=1.5pt] table[x=t,y=Hustar] {data/static_unif_Hamiltonian.txt};
      \addplot+[color=blue,mark=none,line width=1.5pt] table[x=t,y=Hu] {data/static_unif_Hamiltonian.txt};
      \legend{{$\Hcal_{\theta^\dagger}(\varphi_1(\gamma(t)))$},{$\Hcal_{\theta^\dagger}(u^\dagger(t))$}};
      \nextgroupplot[ylabel={$e_\Hcal(t)$},
        xlabel={$t$},
        ylabel style = {yshift=.1cm},
        axis line style = thick,
        grid=both,
        minor tick num=0,
        ymode = log,
        grid style = {gray,opacity=0.2},
        xmin=0, xmax=7,
        ymin=1e-05, ymax=10,
        ytick={1e-08,1e-06,1e-04,1e-02,1e-00},
        xlabel style={font=\footnotesize},
        ylabel style={font=\footnotesize},
        x tick label style={font=\footnotesize},
        y tick label style={font=\footnotesize}]
      \addplot+[color=black,mark=none,line width=1.5pt] table[x=t,y=EH] {data/static_unif_Hamiltonian.txt};
    \end{groupplot}
  \end{tikzpicture}
  \caption{\footnotesize Reconstruction with static, uniformly placed sensors. Left: Hamiltonian evaluated at the exact solution $u^\dagger(t)$ and at the reconstructed solution $\varphi_1(\gamma(t))$. Right: evolution of the error in the Hamiltonian conservation $e_\Hcal(t)$ \eqref{eq:Ham_error}.}\label{fig:static_unif_Hamiltonian}
\end{figure}
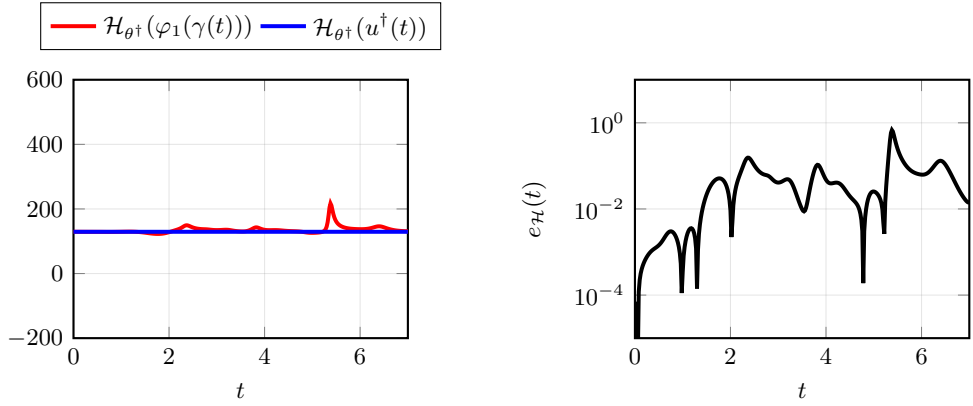

\paragraph{Dynamic sensors.}
The quality of the reconstruction improves significantly when the sensors are allowed to move so that the value of the stability constant $\beta(\Vtwon(t),\Wtwom(t))$ is maximized. First, we show in \Cref{fig:dynamic_solution} the reconstructed solution at the final time obtained with dynamic sensors, and we observe a good qualitative agreement of both components $h$ and $\Phi$ of the reconstruction with the target function associated with the exact parameter.
\begin{figure}[H]
  \centering
  \makebox[\textwidth]{
    \includegraphics[width=1.2\linewidth]{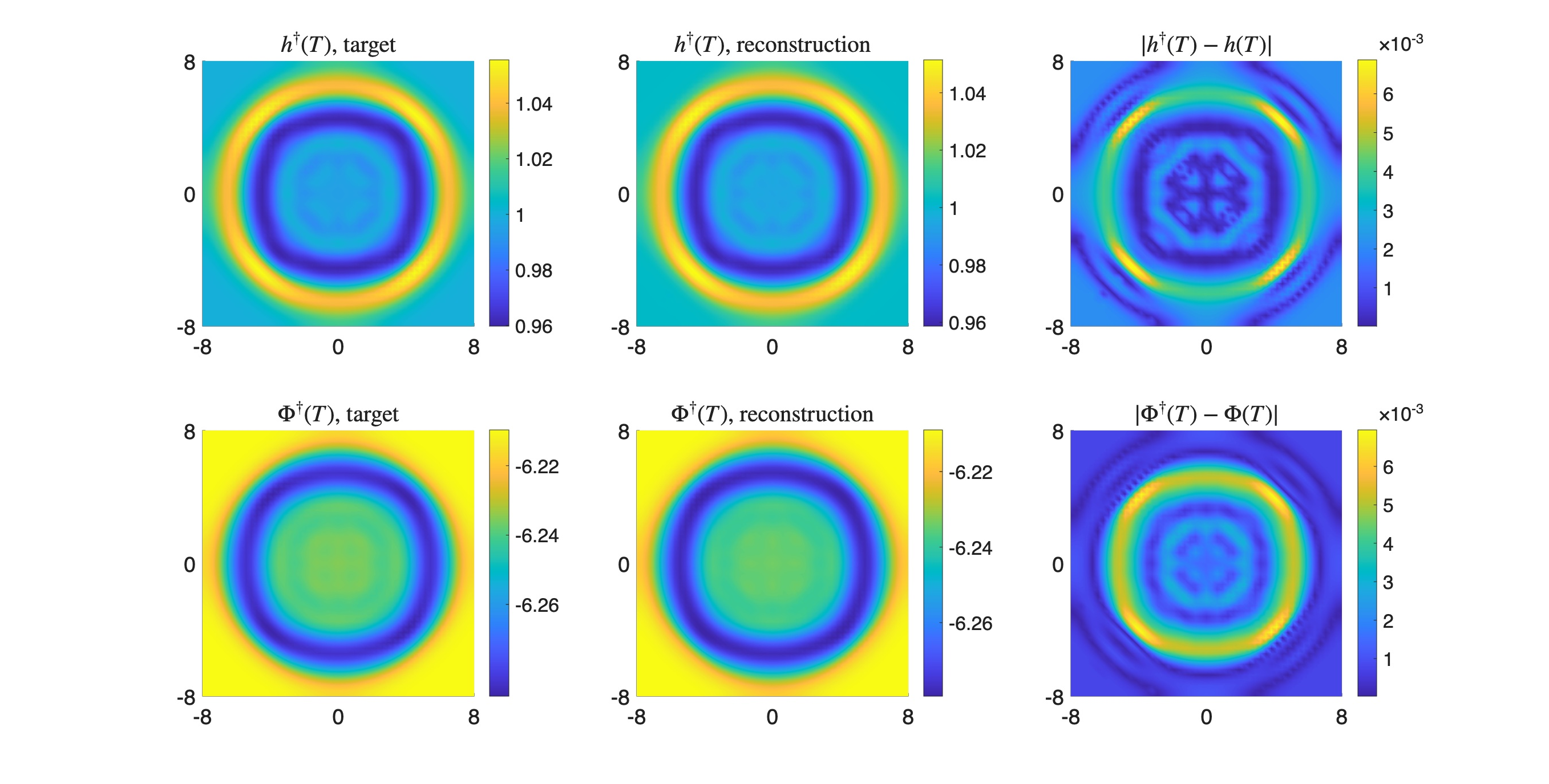}
  }
  \caption{\footnotesize Reconstruction with dynamic sensors. Comparison between the height $h^\dagger(T)$ and the velocity potential $\Phi^\dagger(T)$ associated with the exact parameter at the final time (left column) and the reconstructed quantities (center column), and absolute value of the difference between target and reconstruction (right column).}
  \label{fig:dynamic_solution}
\end{figure}

\Cref{fig:dynamic_error_beta} shows the evolution of the reconstruction error and of $\beta(\Vtwon(t),\Wtwom(t))$: the latter stays above $10^{-2}$ until the final time, while the former is close to the best approximation error at all times and is around $10^{-4}$ at the final time, that is, almost two orders of magnitude lower than in the case of static sensors. We also report the location of the sensors at four time instants during the simulation in \Cref{fig:dynamic_sensors}.
\begin{figure}[H]
  \centering
  \begin{tikzpicture}
    \begin{groupplot}[
        group style={group size=2 by 1,
            horizontal sep=3cm},
        width=6cm, height=5cm,
      ]
      \nextgroupplot[
        xlabel={$t$},
        ylabel style = {yshift=-.2cm},
        axis line style = thick,
        grid=both,
        minor tick num=0,
        minor tick style={draw=none},
        ytick={1e-08,1e-06,1e-04,1e-02,1e+00},
        ymode=log,
        grid style = {gray,opacity=0.2},
        xmin=0, xmax=7,
        ymin=1e-06, ymax=1e-00,
        xlabel style={font=\footnotesize},
        ylabel style={font=\footnotesize},
        x tick label style={font=\footnotesize},
        y tick label style={font=\footnotesize},
        legend style={font=\footnotesize},
        legend cell align={left},
        legend columns = 3,
        legend style={at={(0.5,1.25)},anchor=north}]
      \addplot+[color=red,mark=none,line width=1.5pt] table[x=t,y=e_phi1] {data/dynamic_error_beta.txt};
      \addplot+[color=blue,mark=none,line width=1.5pt] table[x=t,y=e_phi1_und] {data/dynamic_error_beta.txt};
      \addplot+[color=ForestGreen,mark=none,line width=1.5pt] table[x=t,y=e_phi2] {data/dynamic_error_beta.txt};
      \legend{{$e_{\varphi_1}(t)$},{$\underline{e_{\varphi_1}}(t)$},{$e_{\varphi_2}(t)$}};
      \nextgroupplot[ylabel={$\beta(\Vtwon(t),\Wtwom(t))$},
        xlabel={$t$},
        ylabel style = {yshift=.1cm},
        ytick={1e-08,1e-06,1e-04,1e-02,1e+00},
        ymode=log,
        axis line style = thick,
        grid=both,
        minor tick num=0,
        grid style = {gray,opacity=0.2},
        xmin=0, xmax=7,
        ymin=1e-06, ymax=1e+00,
        xlabel style={font=\footnotesize},
        ylabel style={font=\footnotesize},
        x tick label style={font=\footnotesize},
        y tick label style={font=\footnotesize}]
      \addplot+[color=black,mark=none,line width=1.5pt] table[x=t,y=beta] {data/dynamic_error_beta.txt};
    \end{groupplot}
  \end{tikzpicture}
  \caption{\footnotesize Reconstruction with dynamic sensors. Left: evolution of the relative reconstruction error $e_{\varphi_1}(t)$ and the relative approximation error $e_{\varphi_2}(t)$ from \eqref{eq:rec_error}, and the best approximation error $\underline{e_{\varphi_1}}(t)$ from \eqref{eq:bestappr_error}. Right: evolution of the stability constant $\beta(\Vtwon(t),\Wtwom)$.}\label{fig:dynamic_error_beta}
\end{figure}
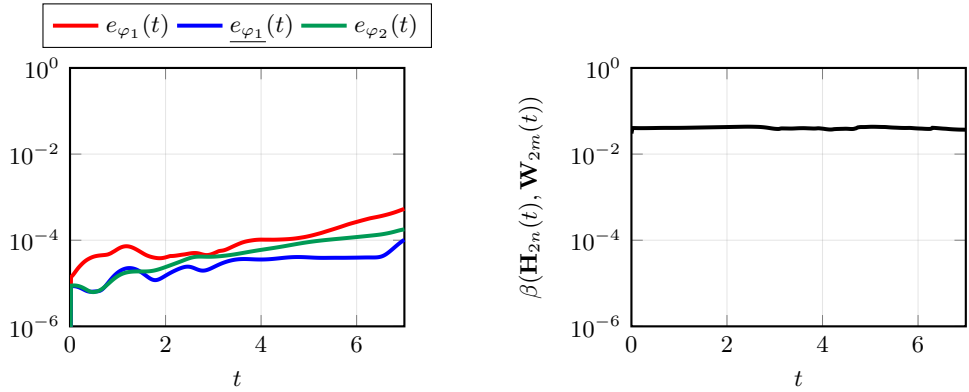

\begin{figure}[H]
  \centering
  \includegraphics[width=0.8\linewidth]{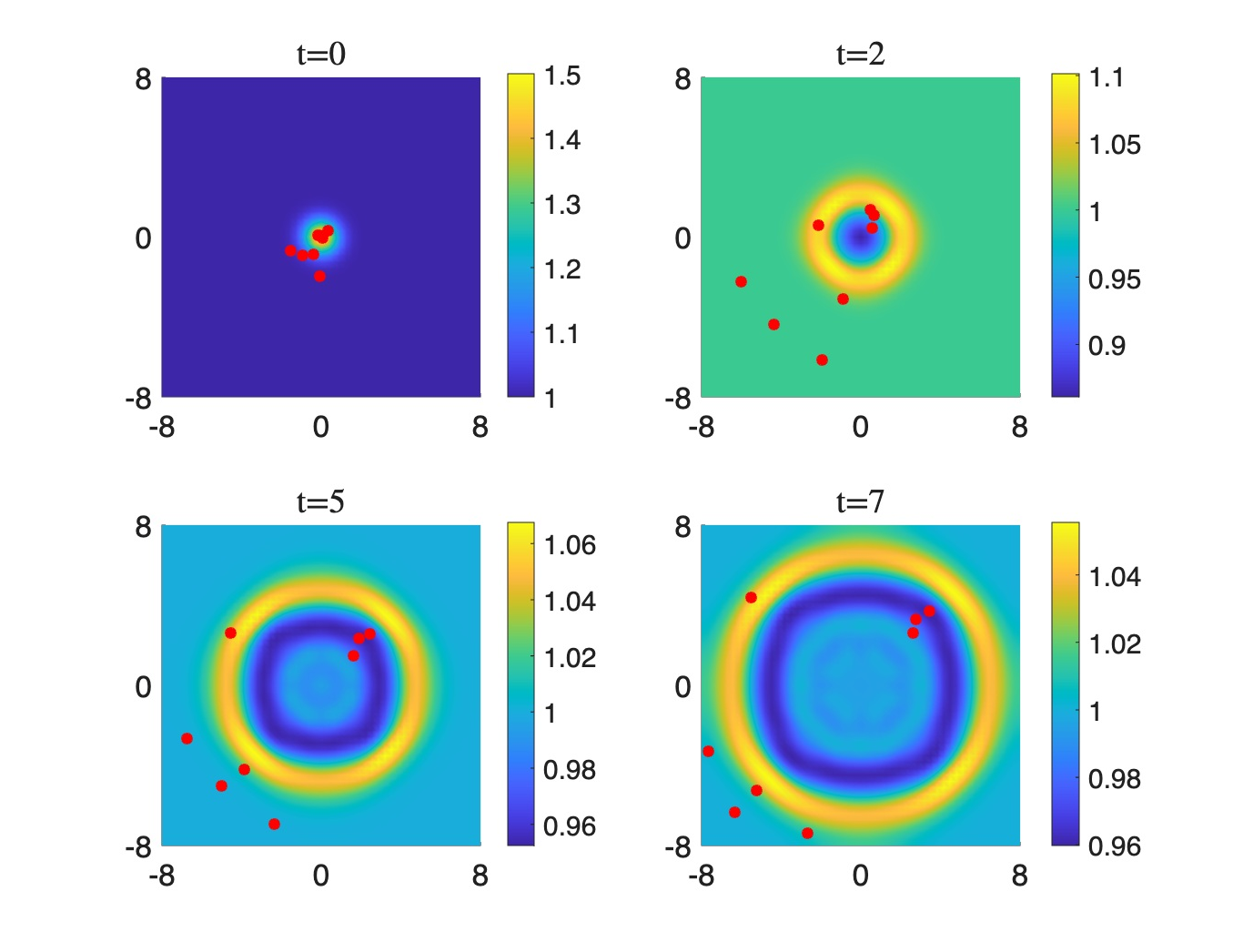}
  \caption{\footnotesize Reconstruction with dynamic sensors. Reconstructed height $h^\dagger(t)$ and location of the sensors (red dots) at four different time instants.}\label{fig:dynamic_sensors}
\end{figure}

Finally, we examine the evolution of the Hamiltonian, which is shown in \Cref{fig:dynamic_Hamiltonian}. We notice that the motion of the sensors allows for a better control of the Hamiltonian error $e_\Hcal$, which remains below $10^{-3}$ until the final stages of the simulation.

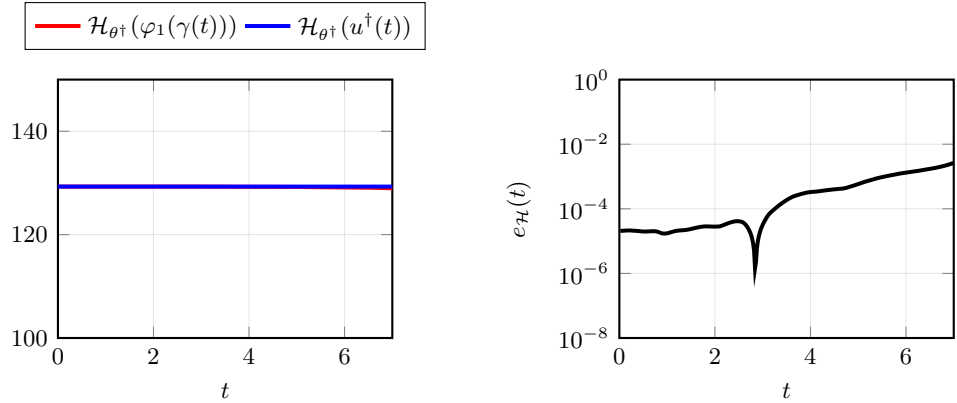
\begin{figure}[H]
  \centering
  \begin{tikzpicture}
    \begin{groupplot}[
        group style={group size=2 by 1,
            horizontal sep=3cm},
        width=6cm, height=5cm,
      ]
      \nextgroupplot[ylabel={},
        xlabel={$t$},
        ylabel style = {yshift=-.2cm},
        axis line style = thick,
        grid=both,
        minor tick num=0,
        minor tick style={draw=none},
        grid style = {gray,opacity=0.2},
        xmin=0, xmax=7,
        ymin=100, ymax=150,
        xlabel style={font=\footnotesize},
        ylabel style={font=\footnotesize},
        x tick label style={font=\footnotesize},
        y tick label style={font=\footnotesize},
        legend style={font=\footnotesize},
        legend cell align={left},
        legend columns = 2,
        legend style={at={(0.5,1.3)},anchor=north}]
      \addplot+[color=red,mark=none,line width=1.5pt] table[x=t,y=Hustar] {data/dynamic_Hamiltonian.txt};
      \addplot+[color=blue,mark=none,line width=1.5pt] table[x=t,y=Hu] {data/dynamic_Hamiltonian.txt};
      \legend{{$\Hcal_{\theta^\dagger}(\varphi_1(\gamma(t)))$},{$\Hcal_{\theta^\dagger}(u^\dagger(t))$}};
      \nextgroupplot[ylabel={$e_\Hcal(t)$},
        xlabel={$t$},
        ylabel style = {yshift=.1cm},
        axis line style = thick,
        grid=both,
        minor tick num=0,
        ymode = log,
        grid style = {gray,opacity=0.2},
        xmin=0, xmax=7,
        ymin=1e-08, ymax=1,
        ytick={1e-08,1e-06,1e-04,1e-02,1e-00},
        xlabel style={font=\footnotesize},
        ylabel style={font=\footnotesize},
        x tick label style={font=\footnotesize},
        y tick label style={font=\footnotesize}]
      \addplot+[color=black,mark=none,line width=1.5pt] table[x=t,y=EH] {data/dynamic_Hamiltonian.txt};
    \end{groupplot}
  \end{tikzpicture}
  \caption{\footnotesize Reconstruction with dynamic sensors. Left: Hamiltonian evaluated at the exact solution $u^\dagger(t)$ and at the reconstructed solution $\varphi_1(\gamma(t))$. Right: evolution of the error in the Hamiltonian conservation $e_\Hcal(t)$ from \eqref{eq:Ham_error}.}\label{fig:dynamic_Hamiltonian}
\end{figure}

\section{Conclusions}
\label{sec:conclusion}
In this work we presented an adaptive framework for reconstructing time-dependent solutions of parametric evolution problems from a finite number of linear measurements, with a particular focus on the case of Hamiltonian systems. The proposed methodology incorporates the information obtained from the measurements into the online evolution of the approximation space.
A rigorous error analysis established the central role of the stability constant $\beta$ in the reconstruction accuracy, and motivated the development of a dynamical strategy to adapt the location of the sensors building on previous work by the authors.

Beyond its data-driven nature, the proposed framework is able to incorporate intrinsic physical properties of the underlying Hamiltonian dynamics. We showed that the reconstruction belongs to a symplectic space, and it preserves the energy up to the approximation error. Numerical experiments on the two-dimensional shallow water equations demonstrated the practical effectiveness of the method and illustrated the benefits of adaptive sensor placement over strategies based on static observations.

Several directions remain open for future investigation. A natural extension of the current work is to augment the framework with a parameter estimation procedure, which would allow one to evolve the probability distribution over the parameter space. Another possible direction is the extension of our methodology beyond linear observations to account for nonlinear measurement operators.

\section*{Acknowledgments}
Co-funded by the European Union (ERC, COCOA, 101170147). Views and opinions expressed are however those of the author(s) only and do not necessarily reflect those of the European Union or the European Research Council. Neither the European Union nor the granting authority can be held responsible for them.

\bibliographystyle{plain}
\bibliography{references}

\appendix
\section{Proof of \Cref{thm:dynamics-gamma-2}}\label{app:proof-gamma-dyn}

We first state some auxiliary lemmas. The first one applies to any pair of subspaces $\mathbf{V}$ and $\mathbf{W}$ of $\fH$ (not necessarily symplectic), and it provides an expression for the orthogonal projection onto $\mathbf{W}\cap\mathbf{V}^\perp$.
\begin{lemma}\label{lem:PWVperp}
  Let $\mathbf{V}$ and $\mathbf{W}$ be finite dimensional subspaces of a Hilbert space $\fH$, and let $V=(v_1,\dots,v_n)$ and $(w_1,\dots,w_m)$ be bases of $\mathbf{V}$ and $\mathbf{W}$, respectively. Assume that $\fV\cap\fW^\top=\{0\}$. Then
  \begin{equation*}
    \proj{\fW\cap\fV^\perp} u = \proj{\fW} u - \sum_{i,j=1}^{n}\Gbb(V,\proj{\fW}V)^{-1}_{i,j}\inner{u, \proj{\fW} v_j}_{\fH}\proj{\fW} v_i \qquad \forall u\in\fH.
  \end{equation*}
\end{lemma}
\begin{proof}
  The assumption that $\fV\cap\fW^\top=\{0\}$ ensures that the matrix $\Gbb(V,\proj{\fW}V)$ is invertible.
  Since $\fW\cap\fV^\perp$ is a closed subspace of $\fH$, every element $u\in\fH$ has a unique projection onto $\fW\cap\fV^\perp$. Let $\widetilde{u}=\proj{\fW} u - \sum_{i,j=1}^{n}\Gbb(V,\proj{\fW}V)^{-1}_{i,j}\inner{u, \proj{\fW} v_j}_{\fH}\proj{\fW} v_i$. Clearly $\widetilde{u}\in\fW$. Moreover, $\widetilde{u}\in\fV^\perp$ since
  \begin{equation*}
    \inner{\widetilde{u}, v_k}_{\fH} = \inner{\proj{\fW}u, v_k}_{\fH} - \sum_{j=1}^{n}(\Gbb(V,\proj{\fW}V)\Gbb(V,\proj{\fW}V)^{-1})_{j,k}\inner{u,\proj{\fW}v_j}_{\fH} = 0 \qquad \forall 1\leq k\leq n.
  \end{equation*}
  Therefore $\widetilde{u}\in\fW\cap\fV^\perp$. Now let $z\in\fW\cap\fV^\perp$. We have that
  \begin{equation*}
    \inner{u-\widetilde{u},z}_{\fH} = \inner{u-\proj{\fW}u, z}_{\fH} + \sum_{i,j=1}^n\Gbb(V,\proj{\fW}V)_{i,j}^{-1}\inner{u, \proj{\fW}v_j}_{\fH}\inner{\proj{\fW}v_i, z}_{\fH} = 0
  \end{equation*}
  since $z\in W$ and $\inner{\proj{\fW}v_i, z}_{\fH}=\inner{v_i, \proj{\fW}z}_{\fH}=\inner{v_i, z}_{\fH}=0$. Since this holds for an arbitrary $z\in\fW\cap\fV^\perp$, we conclude that $\widetilde{u} = \proj{\fW\cap\fV^\perp}u$.
\end{proof}

Next we show that the operator $\cJ$ defined in \eqref{eq:J} commutes with orthogonal projections onto symplectic vector spaces.
\begin{lemma}\label{lem:Jcommutes}
  Let $\Vtwon$ be a symplectic vector subspace of a Hilbert space $\fH$. Then
  \begin{equation*}
    \cJ\circ\proj{\Vtwon} = \proj{\Vtwon}\circ\cJ.
  \end{equation*}
\end{lemma}
\begin{proof}
  Let $V=(v_1,\dots,v_{2n})$ be a $\fH$-orthonormal basis of $\Vtwon$. Using the fact that $\cJ$ is skew-symmetric and $v_i=\cJ v_{i+n}$ for any $1\leq i\leq n$, we have that, for any $u\in\fH$,
  \begin{equation*}
    \begin{aligned}
      \cJ\proj{\Vtwon} u
       & =	\sum_{i=1}^{2n}\inner{u,v_i}_{\fH} \cJ v_i
      = \sum_{i=1}^{n}\prt*{-\inner{u,\cJ v_{i+n}}_{\fH} v_{i+n}-\inner{u,\cJ v_{i}}_{\fH} v_i} \\
       & =\sum_{i=1}^{2n}-\inner{u,\cJ v_i}_{\fH} v_i
      =\sum_{i=1}^{2n}\inner{\cJ u,v_i}_{\fH} v_i = \proj{\Vtwon}\cJ u.
    \end{aligned}
  \end{equation*}
\end{proof}

Finally, we show that $\cJ$ also commutes with orthogonal projections onto the intersection $\Wtwom\cap\Vtwon^\perp$ when $\Wtwom$ and $\Vtwon$ are symplectic vector spaces.
\begin{lemma}\label{lem:Jcommutes_2}
  Let $\Vtwon$ and $\Wtwom$ be symplectic vector subspaces of a Hilbert space $\fH$. Then
  \begin{equation*}
    \cJ\circ\proj{\Wtwom\cap\Vtwon^\perp} = \proj{\Wtwom\cap\Vtwon^\perp} \circ\cJ
  \end{equation*}
\end{lemma}

\begin{proof}
  Let $V=(v_1,\dots,v_{2n})$ be an orthosymplectic basis of $\Vtwon$. Recall that this implies that $\cJ v_j=\sum_{k=1}^{2n}v_k(\bJ_{2n})_{k,j}$ for all $1\leq j\leq 2n$. Using \Cref{lem:PWVperp}, \Cref{lem:Jcommutes} and orthosymplecticity of $V$ we obtain
  \begin{align*}
    \cJ\proj{\Wtwom\cap\Vtwon^\perp}u & = \cJ\proj{\Wtwom}u-\sum_{i,j=1}^{2n}(\Gbb(V,\proj{\Wtwom}V)^{-1})_{i,j}\inprodV{u}{\proj{\Wtwom}v_j}\cJ\proj{\Wtwom}v_i \\& = \proj{\Wtwom}\cJ u-\sum_{i,j=1}^{2n}(\Gbb(V,\proj{\Wtwom}V)^{-1})_{i,j}\inprodV{\cJ u}{\proj{\Wtwom}\cJ v_j}\proj{\Wtwom}\cJ v_i \\&= \proj{\Wtwom}\cJ u-\sum_{k,\ell=1}^{2n}(\bJ_{2n}^\top\Gbb(V,\proj{\Wtwom}V)^{-1}\bJ_{2n})_{k,\ell}\inprodV{\cJ u}{\proj{\Wtwom} v_\ell}\proj{\Wtwom} v_k \\&= \proj{\Wtwom}\cJ u-\sum_{k,\ell=1}^{2n}(\Gbb(V,\proj{\Wtwom}V)^{-1})_{k,\ell}\inprodV{\cJ u}{\proj{\Wtwom}v_\ell}\proj{\Wtwom}v_k \\&= \proj{\Wtwom\cap\Vtwon^\perp}\cJ u
  \end{align*}
  for a generic $u\in\fH$. In the second to last equality we used the fact that $\bJ_{2n}^\top\Gbb(V,\proj{\Wtwom}V)\bJ_{2n}=\Gbb(V,\proj{\Wtwom}V)$, which follows from the definition of $\Gbb(V,\proj{\Wtwom}V)$ \eqref{eq:GVPWV}, orthosymplecticity of $V$ and \Cref{lem:Jcommutes}.
\end{proof}

\begin{proof}[Proof of \Cref{thm:dynamics-gamma-2}]
  We show that the system of equations \eqref{eq:dynamics-dot-2} follows from the system \eqref{eq:dynamics-dot} when the observation space $\Wtwom$ is a symplectic vector space. In this section we omit any dependence on time for simplicity of notation.

  As first thing, we isolate in \eqref{eq:dynamics-dot-V-2} all terms that depend on $\dot{V}$. To this aim we re-write $\chi$ in \eqref{eq:chi} using the expression for $\dot{a}$ in \eqref{eq:dynamics-dot-a}; thereby
  \begin{equation}\label{eq:adot-1}
    \begin{aligned}
      \dot{a}_i & = \sum_{j=1}^{2n}(\Gbb(V,\proj{\Wtwom} V)^{-1})_{i,j}\left(\inner{\proj{\Wtwom}\dot{u}^\dagger, v_j}_{\fH} - \inner{v_j, \proj{\Wtwom}\Big(\sum_{k=1}^{2n}a_k\dot{v}_k\Big)}_{\fH}\right) \\
                & = \sum_{j=1}^{2n}(\Gbb(V,\proj{\Wtwom} V)^{-1})_{i,j}\inner*{\dot{u}^\dagger-\sum_{k=1}^{2n}a_k\dot{v}_k, \proj{\Wtwom} v_j}_{\fH}.
    \end{aligned}
  \end{equation}
  Using this expression and \Cref{lem:PWVperp} we can write
  \begin{align*}
    \proj{\Vtwon^\perp}\proj{\Wtwom}\left(\sum_{i=1}^{2n}\dot{a}_iv_i\right) & = \proj{\Vtwon^\perp}\sum_{i,j=1}^{2n}(\Gbb(V,\proj{\Wtwom} V)^{-1})_{i,j}\inner{\dot{u}^\dagger - \sum_{k=1}^{2n}a_k\dot{v}_k, \proj{\Wtwom} v_j}_{\fH}\proj{\Wtwom} v_i               \\
                                                                             & =\proj{\Vtwon^\perp}\left(\proj{\Wtwom} - \proj{\Wtwom\cap\Vtwon^\perp}\right)\left(\dot{u}^\dagger - \sum_{k=1}^{2n}a_k\dot{v}_k\right)                                                \\
                                                                             & = \proj{\Vtwon^\perp}\proj{\Wtwom}\left(\dot{u}^\dagger - \sum_{k=1}^{2n}a_k\dot{v}_k\right) - \proj{\Wtwom\cap\Vtwon^\perp}\left(\dot{u}^\dagger - \sum_{k=1}^{2n}a_k\dot{v}_k\right).
  \end{align*}
  We then use this identity to expand the definition of the quantity $\chi$ defined in \eqref{eq:chi}:
  \begin{equation*}
    \chi
    =a \proj{\Wtwom\cap\Vtwon^\perp}\dot{u}^\dagger - aa^\top\proj{\Wtwom\cap\Vtwon^\perp}\dot{V} + \lambda\proj{\Vtwon^\perp}\int_{\fTheta} c(\theta)\cJ\prt{\grad_{\fH} \cH_\theta(\varphi_2)}\,\mu(\d\theta).
  \end{equation*}
  We recall that we are using the notation $a^\top V:=\sum_{i=1}^{2n}a_iv_i$ for generic $a\in\bR^{2n}$ and $V\in\fH^{2n}$.
  Moreover, using \Cref{lem:Jcommutes}, \Cref{lem:Jcommutes_2} and the symplecticity of $V\in\cV_{2n}$,
  \begin{align*}
    \bJ_{2n}\cJ s = \bJ_{2n}a\big(\cJ\proj{\Wtwom\cap\Vtwon^\perp}\dot{u}^\dagger - (\bJ_{2n}a)^\top\proj{\Wtwom\cap\Vtwon^\perp}\dot{V}\big) - \lambda\proj{\Vtwon^\perp}\int_{\fTheta}\bJ_{2n}c(\theta)\grad_{\fH} \cH_\theta(\varphi_2)\,\mu(d\theta).
  \end{align*}
  Combining the last two expressions
  and inserting into the evolution \eqref{eq:dynamics-dot-V} of $V$
  yields
  \begin{equation}\label{eq:Vdot-1}
    \bS(c)\dot{V} + \lambda^{-1}\bM(a)\proj{\Wtwom\cap\Vtwon^\perp}\dot{V} =  f(\gamma) + g(\gamma,\dot{u}^{\dagger}),
  \end{equation}
  where $\bM(a):=aa^\top + \bJ_{2n}aa^\top \bJ_{2n}^\top$, $g$ is defined as
  \begin{equation}\label{eq:g}
    g_i(\gamma,\dot{u}^{\dagger}):=\lambda^{-1}\left(a_i + (\bJ_{2n}a)_i \cJ\right)\proj{\Wtwom\cap\Vtwon^\perp}\dot{u}^\dagger \quad \forall 1\leq i\leq 2n,
  \end{equation}
  and $f$ is given in \eqref{eq:f}.
  From the last equation we can obtain an expression for $\dot{V}$ as follows: we write $\dot{V} = \proj{\Wtwom\cap\Vtwon^\perp}\dot{V} + \proj{(\Wtwom\cap\Vtwon^\perp)^\perp}\dot{V}$ and we project \eqref{eq:Vdot-1} onto $\Wtwom\cap\Vtwon^\perp$ and $(\Wtwom\cap\Vtwon^\perp)^\perp$, to find
  \begin{align*}
     & \left(\bS(c)+\lambda^{-1}\bM(a)\right)\proj{\Wtwom\cap\Vtwon^\perp}\dot{V} = \proj{\Wtwom\cap\Vtwon^\perp} \big(f(\gamma) + g(\gamma,\dot{u}^{\dagger})\big), \\
     & \bS(c)\proj{(\Wtwom\cap\Vtwon^\perp)^\perp}\dot{V} = \proj{(\Wtwom\cap\Vtwon^\perp)^\perp} \big(f(\gamma) + g(\gamma,\dot{u}^{\dagger})\big).
  \end{align*}
  If $\bS(c)$ is nonsingular for any $c$, then $\bS(c)+\lambda^{-1}\bM(a)$ is nonsingular for any $\lambda>0$. In particular, with $\alpha(\lambda,a):=\lambda+a^\top \bS(c)^{-1}a$, the Woodbury matrix identity gives
  \begin{equation*}
    \proj{\Wtwom\cap\Vtwon^\perp}\dot{V} =
    \big(\bS(c)^{-1}- \alpha(\lambda,a)^{-1} \bS(c)^{-1}\bM(a)\bS(c)^{-1}\big)
    \proj{\Wtwom\cap\Vtwon^\perp} \big(f(\gamma) + g(\gamma,\dot{u}^{\dagger})\big).
  \end{equation*}
  Hence, since $\proj{\Wtwom\cap\Vtwon^\perp}g = g$, it holds
  \begin{equation}\label{eq:Vdot_temp}
    \dot{V} =
    \bS(c)^{-1} \big(f(\gamma) + g(\gamma,\dot{u}^{\dagger})\big)-\alpha(\lambda,a)^{-1} \bS(c)^{-1}\bM(a)\bS(c)^{-1}\big(\proj{\Wtwom\cap\Vtwon^\perp} f(\gamma) + g(\gamma,\dot{u}^{\dagger})\big).
  \end{equation}

  It can be easily verified that the matrix $\Sbb(c)=\Cbb(c)+\bJ_{2n}^\top\Cbb(c)\bJ_{2n}$ is symmetric and skew-Hamiltonian for all $c$, that is, $\Sbb(c)^\top = \Sbb(c)$ and $\bJ_{2n}\Sbb(c) = \Sbb(c)\bJ_{2n}$.
  These properties imply that
  $v^\top\bS(c)^{-1}\bJ_{2n} v = 0$ for any $v\in\bR^{2n\times 2n}$,
  which, together with the definition of $\bM(a)$, gives
  \begin{equation*} \bM(a)\bS(c)^{-1}(a+(\bJ_{2n}a)\cJ) = (a^\top\bS(c)^{-1}a)(a+(\bJ_{2n}a)\cJ).
  \end{equation*}
  Inserting
  the expression of $g$ given in \eqref{eq:g} into \eqref{eq:Vdot_temp} and using the latter identity yields
  \begin{equation*}
    \dot{V} = \bS(c)^{-1}\left(
    f(\gamma)+\alpha(\lambda,a)^{-1}
    (a+(\bJ_{2n}a)\cJ)\proj{\Wtwom\cap\Vtwon^\perp}\dot{u}^\dagger - \alpha(\lambda,a)^{-1}\bM(a)\bS(c)^{-1}\proj{\Wtwom\cap\Vtwon^\perp} f(\gamma)\right).
  \end{equation*}
  Moreover, using again the properties of $\Sbb(c)$ and the fact that $\cJ f_i = \sum_{j=1}^{2n}(\bJ_{2n}^\top)_{i,j}f_j$ for all $1\leq i\leq2n$, it holds
  \begin{equation*}
    \sum_{i=1}^{2n}\left(\bS(c)^{-1}\bJ_{2n}a\right)_i f_i = \sum_{i=1}^{2n} \left(\bS(c)^{-1}a\right)_i\cJ f_i,
  \end{equation*}
  which allows us to write $\bM(a)\bS(c)^{-1}\proj{\Wtwom\cap\Vtwon^\perp}f = (a+(\bJ_{2n}a)\cJ)\proj{\Wtwom\cap\Vtwon^\perp}(a^\top\bS(c)^{-1}f)$. Hence,
  \begin{equation*}
    \dot{V} = \bS(c)^{-1}\left(f(\gamma) +  \alpha(\lambda,a)^{-1}(a+(\bJ_{2n}a)\Jcal)\proj{\Wtwom\cap\Vtwon^\perp}(\dot{u}^\dagger-a^\top \bS(c)^{-1}f(\gamma))\right),
  \end{equation*}
  and we obtain the desired expression \eqref{eq:dynamics-dot-V-2} for $\dot{V}$.

  Finally, we observe that, for any $u\in\fH$, $\inner{\proj{\Wtwom\cap\Vtwon^\perp}u, \proj{\Wtwom}v_j}_{\fH} = 0$ for all $1\leq j \leq 2n$ since $v_j\in\Vtwon$, and that
  \begin{align*}
    \inner{\Jcal\proj{\Wtwom\cap\Vtwon^\perp}u, \proj{\Wtwom}v_j}_{\fH} & = -\inner{\proj{\Wtwom\cap\Vtwon^\perp}u, \cJ\proj{\Wtwom}v_j}_{\fH} = - \inner{\proj{\Wtwom\cap\Vtwon^\perp}u, \proj{\Wtwom}\cJ v_j}_{\fH} \\&= -\sum_{k=1}^{2n}(\bJ_{2n})_{k,j}\inner{\proj{\Wtwom\cap\Vtwon^\perp}u, \proj{\Wtwom} v_k}_{\fH} = 0, \quad \forall 1\leq j\leq 2n,
  \end{align*}
  where we used \Cref{lem:Jcommutes} and orthosymplecticity of $V$.
  Using the above identities and inserting the expression for $\dot{V}$ into \eqref{eq:adot-1} yields
  \begin{equation*}
    \begin{aligned}
      \dot{a}_i & = \sum_{j=1}^{2n}(\Gbb(V,\proj{\Wtwom} V)^{-1})_{i,j}\inner*{\dot{u}^\dagger-\sum_{k=1}^{2n}a_k(\bS(c)^{-1}f)_k, \proj{\Wtwom} v_j}_{\fH},
    \end{aligned}
  \end{equation*}
  which gives \eqref{eq:dynamics-dot-a-2}.
\end{proof}

\section{Spatial and temporal discretization}\label{app:discrete}
In this Section we derive a discrete formulation of \eqref{eq:dynamics-dot-a-2}-\eqref{eq:dynamics-dot-c-2}-\eqref{eq:dynamics-dot-V-2} that is suitable for numerical simulations. Let us consider the basis functions $v_i(t)=(v_i^q(t),v_i^p(t))\in\fH=\Hhat\times\Hhat$, for all $1\leq i\leq2n$, and a generic spatial discretization scheme in $\Hhat$. Let $N$ be the number of degrees of freedom. We collect the degrees of freedom associated with $v_i^q(t)$ and $v_i^p(t)$ for all $i$ at the generic time $t$ in the matrices $\widetilde{\bU}^q(t)\in\bR^{N\times2n}$ and $\widetilde{\bU}^p(t)\in\bR^{N\times2n}$, respectively.
At the discrete level, orthosymplecticity of $V=(v_1,\dots,v_{2n})$ implies that $\widetilde{\bU}^\top\bM\widetilde{\bU}=\bI_{2n} \text{ and } \widetilde{\bU}^\top\bM \bJ_{2N}\widetilde{\bU} = \bJ_{2n}$, where
$$\widetilde{\bU}(t)=\begin{bmatrix}
    \widetilde{\bU}^q(t) \\ \widetilde{\bU}^p(t)
  \end{bmatrix}\in\bR^{2N\times2n}\qquad \bM=\begin{bmatrix}
    \widehat{\bM} & 0 \\ 0 & \widehat{\bM}
  \end{bmatrix}\in\bR^{2N\times2N}$$
and $\widehat{\bM}\in\bR^{N\times N}$ is the mass matrix associated with the discretization in $\Hhat$. This means that the matrix $\bU:=\bM^{1/2}\widetilde{\bU}$ is orthosymplectic, that is, it belongs to the manifold
\begin{equation*}
  \cM(2n,\bR^{2N}) = \left\{\bU\in\bR^{2N\times2n} : \bU^\top\bU=\bI_{2n} \text{ and } \bU^\top \bJ_{2N}\bU = \bJ_{2n}\right\}.
\end{equation*}

Similarly, we discretize the parameter set $\fTheta$ by a set of $p$ sample parameters $\theta_j$, $1\leq j\leq p$, and we introduce the matrix $\bZ(t)\in\bR^{2n\times p}$ that collects the evaluations of the coefficients $c_i$ at the sample parameters at time $t$: $\bZ(t)_{i,j} = c_i(t)(\theta_j)$. We remark that the condition on the coefficients $c(t)(\theta)$ belonging to the space $\cC_{2n}$ defined in \eqref{eq:space-C2n} translates to
\begin{equation}\label{eq:full-rank-cond-Z}
  \text{rank}(\bZ(t)\bZ(t)^\top + \bJ_{2n}\bZ(t)\bZ(t)^\top\bJ_{2n}^\top) = 2n \qquad \forall t.
\end{equation}
A necessary condition for this full-rank property to be satisfied is that $p\geq2n$. In other words, if the basis is too large compared to the number of sample parameters, one may incur a rank-deficient evolution problem for the matrix $\bZ$. In this work, we assume that the full-rank condition \eqref{eq:full-rank-cond-Z} always holds, and we refer to, e.g. \cite[Section 5]{HPR22}, for a discussion on how to address rank-deficiencies in the dynamics.

We collect the degrees of freedom associated with the Riesz representers of the measurement functionals at time $t$ in a matrix $\widetilde{\bW}(t)\in\bR^{2N\times2m}$.
In the numerical experiments of \Cref{sec:num-exp} we assume that the expression of the Riesz representers spanning the observation space is known, and therefore we can evaluate $\bW(t):=\bM^{1/2}\widetilde{\bW}(t)$ from the locations of the sensors $\overline{x}(t)$ at time $t$. In a more general setting, one would have to numerically compute the Riesz representers by solving a suitable boundary value problem.

Thus, \eqref{eq:dynamics-dot-a-2}-\eqref{eq:dynamics-dot-c-2}-\eqref{eq:dynamics-dot-V-2} yields a system of ODEs for the unknowns $(a(t), \bZ(t), \bU(t))\in\mathbb{R}^{2n}\times\mathbb{R}^{2n\times p}\times\Mcal(2n,\mathbb{R}^{2N})$ of the form
\begin{equation}\label{eq:ODE_discr}
  \begin{cases}
     & \dot{a}(t) = f_1(a(t), \bZ(t), \bU(t); \bW(t), \dot{\rz}(t))   \\
     & \dot{\bZ}(t) = f_2(\bZ(t), \bU(t))                             \\
     & \dot{\bU}(t) = f_3(a(t), \bZ(t), \bU(t); \bW(t), \dot{\rz}(t))
  \end{cases}
\end{equation}

Concerning time discretization of \eqref{eq:ODE_discr}, we subdivide the time interval $(0,T]$ into $N_t$ sub-intervals $(t_j,t_{j+1})$, with $0\leq j\leq N_t-1$, where for simplicity we consider uniform time steps $\Delta t = t_{j+1}-t_j=T/N_t$, and we set out to produce approximations $a_j$, $\bZ_j$ and $\bU_j$ to $a(t_j)$, $\bZ(t_j)$ and $\bU(t_j)$. We assume that the measurements $\dot{\rz}$ can only be acquired at the time instants $t_j$, and that the sensors are fixed in the time sub-interval $(t_j, t_{j+1})$. In other words, we set $\dot{\rz}(t)=\dot{\rz}(t_j)=:\dot{\rz}_j$ and $\bW(t) = \bW(\overline{x}(t)) = \bW(\overline{x}(t_j)) =:\bW_j$ for $t\in(t_j,t_{j+1})$. Therefore, solving \eqref{eq:ODE_discr} numerically in $(t_j,t_{j+1})$ produces approximations $a_{j+1}$, $\bZ_{j+1}$ and $\bU_{j+1}$ at time $t_{j+1}$ according to the generic procedure
\begin{equation*}
  (a_{j+1}, \bZ_{j+1}, \bU_{j+1})\gets\textsc{Evolve}(a_j, \bZ_j, \bU_j; \bW_j, \dot{\rz}_j).
\end{equation*}
In this work, we consider a tangent method based on retraction maps as proposed in \cite{P21}, and we apply a second-order partitioned Runge-Kutta scheme to the resulting formulation. In particular, we combine the implicit midpoint scheme for $\bZ$ with the explicit midpoint method (or modified Euler method) for the evolution of $a$ and for the dynamics on the tangent space to $\mathcal{V}_{2n}$ corresponding to the evolution of $\bU$. This approach ensures that the matrices $\bU_j$ are orthosymplectic at all time steps $t_j$. Moreover, since the implicit midpoint scheme is a symplectic integrator, the Hamiltonian structure of the evolution equation for the coefficients $\bZ$ is preserved at the time-discrete level. We refer to \cite[Section 5.3]{P21} for more details on this procedure.

It remains to prescribe an initial condition $(a_0, \bZ_0, \bU_0)$ at time $t=0$. This can be done by discretizing \eqref{eq:V0}, \eqref{eq:a0} and \eqref{eq:c0}, respectively. At the discrete level, \eqref{eq:V0} reads
\begin{equation}\label{eq:minprob_U0}
  \bU_0 = \argmin_{\bU\in\cM(2n,\bR^{2N})}\norm{\bY_0-\bU\bU^\top\bY_0}_F,
\end{equation}
where $\bY_0=\bM^{1/2}\begin{bmatrix}
    \widetilde{\bY}_0^q \\ \widetilde{\bY}_0^p
  \end{bmatrix}\in\bR^{2N\times p}$ and $\widetilde{\bY}_0^q$ and $\widetilde{\bY}_0^p$ collect the degrees of freedom associated with the initial conditions $u^q(\theta_j)(0)$ and $u^p(\theta_j)(0)$, respectively, for $1\leq j \leq p$. The minimization problem \eqref{eq:minprob_U0} can be solved by means of a complex SVD of $\bY_0$ as detailed in \cite[Section 4.2]{PM16}. Then, the matrix of coefficients at the initial time can then be computed from \eqref{eq:c0} as $\bZ_0=\bU_0^\top \bY_0$. Finally, discretization of \eqref{eq:a0} and \eqref{eq:r0} yields
\begin{equation}\label{eq:minprob_a0}
  \bU_0^\top\bW_0(\bW_0^\top\bW_0)^{-1}\bW_0^\top\bU_0a_0 = \bU_0^\top\bW_0(\bW_0^\top\bW_0)^{-1}\rz_0,
\end{equation}
where $\bW_0:=\bW(0)=\bW(\overline{x}(0))$ and $\rz_0\in\bR^{2m}$ is the vector of measurements at the initial time. Solving \eqref{eq:minprob_a0} yields the initial condition $a_0$.

We can now summarize the entire procedure in the following algorithm.
\begin{algorithm}[H]\label{alg:symp_filt}
  \caption{Symplectic filtering with moving sensors}
  \begin{algorithmic}[1]
    \Procedure{ \textsc{Symplectic\_filtering}}{$\bY_0$, $\overline{x}_0$}
    \State Construct $\bU_0$ via complex SVD of $\bY_0$, and construct $\bZ_0=\bU_0^\top\bY_0$
    \State $\overline{x}_0 \gets \textsc{Sensors update}(\overline{x}_0,  \bU_0)$ as in \Cref{alg:sensors_update}
    \State Define $\bW_0=\bW(\overline{x}_0)$ and obtain measurements $\dot{\rz}_0$
    \State Construct $a_0$ by solving \eqref{eq:minprob_a0}
    \For{$j=0,\dots,N_t-1$}
    \State $(a_{j+1}, \bZ_{j+1}, \bU_{j+1})\gets\textsc{Evolve}(a_j, \bZ_j, \bU_j; \bW_j, \dot{\rz}_j)$
    \State $\overline{x}_{j+1} \gets \textsc{Sensors update}(\overline{x}_j,  \bU_{j+1})$ as in \Cref{alg:sensors_update}
    \State Define $\bW_{j+1}=\bW(\overline{x}_{j+1})$ and obtain measurements $\dot{\rz}_{j+1}$
    \EndFor
    \EndProcedure
  \end{algorithmic}
\end{algorithm}

\end{document}